\documentclass[11pt,letterpaper,reqno]{amsart}
\usepackage[utf8]{inputenc}
\usepackage{geometry}
\usepackage{amsmath,amssymb,amsfonts,amsthm,amscd,hyperref} 
\usepackage{mathrsfs,latexsym,url,setspace,xcolor,graphicx, tikz}
\newtheorem{theorem}{Theorem}
\newtheorem{theoremA}{Theorem}
\newtheorem*{theorem*}{Theorem}

\newtheorem{proposition}{Proposition}
\newtheorem{lemma}{Lemma}

\theoremstyle{remark}

\newtheorem{example}{Example}

\newcommand{\cx}{\mathbb{C}}

\newcommand{\disk}{\mathbb{D}}
\newcommand{\D}{\Omega}

\newcommand{\ep}{\varepsilon}

\newcommand{\Dc}{\overline{\Omega}}
\newcommand{\dbar}{\overline{\partial}}

\newcommand{\norm}[1]{\left\Vert#1\right\Vert}
\newcommand{\abs}[1]{\left\vert#1\right\vert}
\newcommand{\ipr}[1]{\left\langle #1 \right \rangle}

\newcommand{\rl}{{\mathbb{R}}}
\newcommand{\ol}{\overline}
\newcommand{\wt}{\widetilde}

\newcommand{\tmop}[1]{\ensuremath{\operatorname{#1}}}
\renewcommand{\Re}{\tmop{Re}}
\renewcommand{\Im}{\tmop{Im}}

\title{The restriction operator on  Bergman spaces}

\author{Debraj Chakrabarti}
\address[Debraj Chakrabarti]{Department of Mathematics, 
	Central Michigan University, Mt. Pleasant,  MI 48859,  USA}
\email{chakr2d@cmich.edu}
\author{S\"{o}nmez \c{S}ahuto\u{g}lu}
\address[S\"{o}nmez \c{S}ahuto\u{g}lu]{Department of Mathematics 
	and Statistics, University of Toledo, Toledo, OH 43606, USA} 
\email{sonmez.sahutoglu@utoledo.edu}
\thanks{Debraj Chakrabarti was partially supported by  grant from 
	the National Science Foundation] (\#1600371),  a collaboration grant from the 
	Simons Foundation (\# 316632) and also by an Early Career internal 
	grant from Central Michigan University.}


\begin{document}

\begin{abstract}
We study the restriction operator from the Bergman space of a domain in 
$\mathbb{C}^n$ to the Bergman space of a non-empty open subset of the 
domain. We relate the restriction operator to the Toeplitz operator on the 
Bergman space of the domain whose symbol is the characteristic function 
of the subset. Using the biholomorphic invariance of the  spectrum of the 
associated Toeplitz operator, we study the restriction operator from the 
Bergman space of the unit disc to the Bergman space of subdomains with 
large symmetry groups, such as horodiscs and subdomains bounded by 
hypercycles. Furthermore, we prove a sharp estimate of the norm of the 
restriction operator in case the domain and the subdomain are balls. We 
also study various operator theoretic properties of the restriction operator 
such as compactness and essential norm estimates.
\end{abstract}
\maketitle

 \section{Introduction}

 \subsection{The restriction operator} 
 Let $\Omega$ be a domain in $\cx^n$ and $A^2(\Omega)$ be {its}
 \textit{Bergman space}, the linear space of holomorphic functions which are 
 square integrable with respect to the Lebesgue measure. Recall that 
 $A^2(\Omega)$ is a closed subspace of $L^2(\Omega)$ and therefore 
 is a Hilbert space with the induced inner product. If $U\subset \Omega$ 
 is a nonempty open subset, one can define the \textit{restriction operator} 
 $R_U: A^2(\Omega)\to A^2(U)$ given by $R_Uf=f|_U$. Clearly, $R_U$ is 
 a bounded linear map between Hilbert spaces.  In some problems of 
 complex analysis, one may want to know further properties of the operator 
 $R_U$. For example, the non-compactness of $R_U$ in a special situation 
 plays a crucial role in Fu and Straube's seminal study of the compactness 
 of the $\dbar$-Neumann problem on convex domains (see \cite{FuStraube98}). 
  
 The goal of this paper is to study  the operator $R_U$ by establishing a relation 
 with the theory of \textit{Toeplitz operators} on Bergman spaces. Using this relation,
 given any domain $\Omega$ in $\cx^n$, and a non-empty open subset 
 $U\subset \Omega$, we can define a \textit{spectrum} of $U$ relative to 
 $\Omega$ in such a way that this spectrum is an invariant of the  complex 
 geometry of $\Omega$. More  precisely, if $\wt{\Omega}$ is a domain 
 biholomorphic to $\Omega$, and $f:\Omega\to \wt{\Omega}$ is a 
 biholomorphic mapping, then the spectrum of $U$ with respect to 
 $\Omega$ is identical to the spectrum of $f(U)$ with respect to 
$\wt{\Omega}$. This clearly assumes a special interest when 
$f:\Omega\to \Omega$ is a biholomorphic automorphism of $\Omega$, 
and it follows that two subsets of $\Omega$ which are biholomorphically 
congruent (i.e., there is an automorphism of $\Omega$ mapping one to 
the other) are also isospectral. As we will see, the spectrum encodes 
important geometric and function theoretic information about the way 
the set $U$ ``sits inside'' $\Omega$.

 The considerations of this paper extend easily to complex manifolds, 
 provided we use the intrinsic Bergman spaces of square-integrable 
 holomorphic $n$-forms. For ease of exposition we will confine ourselves 
 to domains in $\cx^n$. We will further assume that $A^2(\Omega)$ is 
non-trivial (that is, $A^2(\Omega)\neq \{0\}$), an assumption which is  
certainly satisfied when $\Omega$ is either bounded or biholomorphic 
to a bounded domain. 

\subsection{The associated Toeplitz operator}
 We introduce the  operator
\begin{equation}\label{eq-tu}
 T_U = R^*_U R_U: A^2(\Omega)\to A^2(\Omega),
\end{equation}
where $R^*_U: A^2(U)\to A^2(\Omega)$ is the Hilbert-space adjoint of $R_U$. 
The operator $T_U$ is clearly a positive self-adjoint operator on  $A^2(\Omega)$, 
and it is not difficult to show that it is the \textit{Toeplitz operator} on the 
Bergman space $A^2(\Omega)$  with symbol the characteristic function 
$\chi_U$ of the open subset $U$. That is, for $f\in A^2(\Omega)$:
\[ T_U f =P(\chi_U f),\]
where $P: L^2(\Omega)\to A^2(\Omega)$ is the Bergman projection 
(see Proposition~\ref{prop-t} below).  It is natural to try to determine 
the spectrum of the selfadjoint operator $T_U$, which is precisely the 
spectrum of $U$ relative to $\Omega$ referred to above. As noted above 
the  spectrum of $T_U$ is invariant under biholomorphisms of the ambient 
domain $\Omega$ (see Proposition \ref{prop-invariance} below).

\subsection{Main results} \label{sec-results}
Many properties of $R_U$ can be read off from the spectrum of $T_U$, 
and so a precise understanding of $T_U$ (and therefore $R_U$)
involves of a description of the spectral decomposition of $T_U$. When this 
cannot be done, one might want to  measure the ``size'' of $T_U$ relative to 
the identity operator on $A^2(\Omega)$. This gives us a way of measuring 
the ``function-theoretic size" of $U$ relative to $\Omega$. Among such 
measures are the norm $\norm{T_U}=\norm{R_U}^2$, the essential 
norm $\norm{T_U}_e$ and when $T_U$ is compact (i.e.
$\norm{T_U}_e=0$) the Schatten $p$-norms of $T_U$. Note 
that each of these measures is monotone in the set $U$, i.e., they 
increase with the set $U$.

It is also natural to consider for each open set $U\subset \Omega$ the 
\textit{complementary restriction} $R_{\Omega\setminus{\ol{U}}}$ 
from $A^2(\Omega)$ to the open set $\Omega\setminus \ol{U}$. 
In order to avoid pathological situations, we will make the assumption 
that the boundary $\partial U$ of $U$ has zero Lebesgue measure. 
The following result relates some of these notions.

 \begin{theorem} \label{thm-basic} 
{Let $\D$ be a domain in $\cx^n$ and $U$ be a non-empty open subset of $\D$ 
	such that $A^2(\Omega) $ is infinite dimensional and the  boundary 
	of $U$ has zero Lebesgue measure.} Then 
 \begin{enumerate}
 \item If $R_U$ is compact then $\norm{R_U}<1$ and 
 $\norm{R_{\Omega\setminus \ol{U}}}_e
 =\norm{R_{\Omega\setminus \ol{U}}}=1$.
 \item $\norm{R_U}<1$ if and only if the complementary restriction 
 	$R_{\Omega\setminus\ol{U}}$ has closed range.
 \end{enumerate}
 \end{theorem}   
 
From part (1), it follows that if $\norm{R_U}=1$, then $R_U$ (and therefore 
$T_U$) is non-compact. However, the converse is not true, i.e., there exist 
open sets $U\subset \Omega$ such that $\norm{R_U}<1$ and $R_U$ is 
not compact. Examples for $n=1$ may be found in  part (2) of 
Theorem \ref{thm-horostrip} (``horocyclic strips'') and also part (2) 
of Theorem \ref{thm-hypercycle} (``hypercyclic lunes''). In higher dimensions, 
it is even possible for $U$ to be a smoothly bounded subdomain of $\Omega$ 
touching $\Omega$ at a point. For instance, let $\D$ be a ball and 
$U\subset \D$ be a another ball with a common boundary point with $\D$. 
Then Theorem \ref{thm-balls} below implies that $\norm{R_U}<1$. However, 
$R_U$ is not compact (see \cite[Lemma 4.27]{StraubeBook}).

Part (2) of Theorem \ref{thm-basic} may be thought of as a generalized form 
of \textit{Hartogs phenomenon}. The classical Hartogs phenomenon refers to 
the situation of a domain $\Omega$ and a subset $U\subset \Omega$,
where each   holomorphic function on $\Omega\setminus \ol{U}$ extends 
to a holomorphic function on $\Omega$. This happens, for instance, when  
$U$ is relatively compact in $\Omega$, and the complement $\Omega \setminus \ol{U}$ 
is  connected (see, for instance \cite[Theorem 3.1.2]{ChenShawBook}). Hartogs 
phenomenon can happen only in dimensions $n\geq 2$.  An $L^2$-version of 
Hartogs phenomenon consists of the situation when $R_{\Omega\setminus\ol{U}}$ 
is an isomorphism from $A^2(\Omega)$ to $A^2(\Omega\setminus \ol{U})$. 
In this case $R_{\Omega\setminus\ol{U}}$ has closed range equal to $A^2(\Omega)$, 
and therefore by part (2) of our result we have $\norm{R_U}<1$. 
This suggests that restriction operators of norm less than one are 
more common in dimensions $\geq 2$. This is confirmed in 
Theorem \ref{thm-balls} and Example \ref{ex-hartogsfigure} below.

Theorem \ref{thm-basic} shows that it is important to determine 
when the norm of the restriction operator is strictly less than 1. 
In Section \ref{sec-norm} we look at the problem of computing the norm 
when the boundaries of $\Omega$ and $U$ touch at a smooth point of 
both. Theorem \ref{thm-basic} suggests that the results will be different 
for $n=1$ and $n\geq 2$. For $n=1$, in Proposition \ref{prop-generalplanar}, 
we show that  $\norm{R_U}=1$, if $\partial\Omega$ and $\partial U$ touch 
at a $C^1$-smooth point. Note that if we consider a $U$ such that a non-smooth 
point of $\partial U$ touches $\partial\Omega$, then even in one dimension, 
we may have $\norm{R_U}<1$ (see e.g. part (2) of Theorem~\ref{thm-hypercycle} 
below).

 We then consider the case when  $U$ is a ball contained in the ball 
 $\Omega$. For $n=1$ one can get an exact value of $\|R_U\|$ from 
 Proposition \ref{prop-compactdisc} when $\overline{U}\subset \D$, 
 where in fact we obtain the full spectral decomposition. 
When $n=1$ and $\partial U$ and $\partial \Omega$ touch at a point, 
the spectral decomposition of $T_U$ will be worked out in 
Theorem \ref{thm-horostrip} below. The case $n\geq 2$ is 
considered in the following.

\begin{theorem}\label{thm-balls}
Let $\Omega$ and $U$ be two balls in $\cx^n$  of radii $R$ and $r$, respectively, 
such that $U\subset \Omega$, and let $\delta$ denote the distance from the 
center of $U$ to the boundary of $\Omega$. Then 
\begin{equation} \label{eq-ruestimate}
\left(\frac{r}{R}\right)^n \leq \norm{R_U} 
\leq \left(\frac{\delta}{R}\right)^{\frac{n-1}{2}}. 
\end{equation} 
The inequalities are sharp in the sense that the lower bound is attained when 
$U$ and $\Omega$ are concentric and the upper bound is attained when the 
boundaries of the two balls touch at a point.
\end{theorem}

As is apparent from the examples in Section~\ref{sec-example} below, 
when the ambient domain $\Omega$ and the subset $U$ have a large 
common symmetry group, the spectrum of $T_U$ can sometimes be 
determined  explicitly. In this case, where we have Reinhardt symmetry, 
the operator $T_U$ has pure point spectrum. In Section \ref{sec-horosymmetry},  
we give a few other examples of subdomains of the unit disc for which 
there is sufficient symmetry to compute explicitly the spectrum of $T_U$. 
These computations are based on ideas of Vasilevski (see \cite{vasilevski} 
and references therein) on spectral representation of algebras of Toeplitz 
operators. Below we describe these examples in detail, which apart from 
their intrinsic interest also provide simple counterexamples for conjectures 
related to the spectrum of $T_U$.

Let $\disk$ be the unit disc in the plane, which we can think of as the hyperbolic 
plane with the Poincaré metric. Recall from hyperbolic geometry (see 
\cite[Chapter XI]{coxeter} for more information)  that a \textit{horodisc} in 
in the hyperbolic plane is represented in the Poincaré disc 
$\disk$ by an Euclidean disc whose boundary is tangent to the unit 
circle. Intuitively, a horodisc is a ``disc of infinite radius'' in the hyperbolic plane.
The boundary of a horodisc is called a \textit{horocycle}, and  the open set between 
two horocycles  tangent to $\partial \disk$ at the same point will be called a 
\textit{horocyclic strip}. Therefore, for each $\alpha\in \disk$, the set 
$ \{z\in\cx:\abs{z-\alpha}< 1- \abs{\alpha}\}$ is a horodisc, and its boundary 
in $\disk$, $ \{z\in \cx:\abs{z-\alpha}= 1- \abs{\alpha}, \abs{z}\not=1\}$,  
is a horocycle. We naturally call $\sqrt{1-\abs{\alpha}}$ the \textit{Euclidean radius} 
of the horocycle. The shaded region in Figure \ref{fig-horostrip} shows 
the horocyclic strip
\[ \left\{z\in \cx:\abs{z-\frac{1}{2}} < \frac{1}{2}, 	
\abs{z-\frac{3}{4}}> \frac{1}{4}\right\}.\]
\begin{figure}
\begin{center}
\begin{tikzpicture}[scale=2.5]
\draw (0,0) circle (1cm); 
\draw [fill=black!40] (0.5,0) circle (0.5cm);
\draw[fill=white] (0.75, 0) circle( 0.25cm);
\end{tikzpicture}
\end{center}
\caption{Horocyclic Strip}\label{fig-horostrip}
\end{figure}
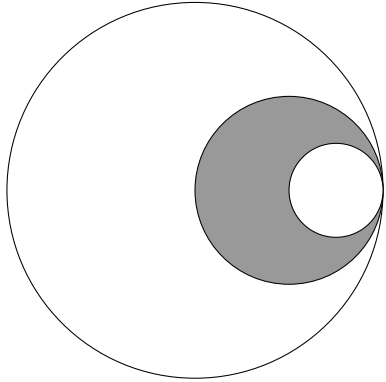

In the following theorem we characterize the spectrum of the restriction operator 
in case $U$ is a horodisc or a horocyclic strip in the unit disc.

\begin{theorem} \label{thm-horostrip} 
	Let $\disk$ be the unit disc in $\cx$. 
\begin{enumerate}
\item If $U\subset \disk$ is a horodisc, then the spectrum of $T_U$ is the 
	full interval $[0,1]$.
\item Let $0<\rho_1<\rho_2<1$ and $U\subset \disk$ be a horocyclic strip 
	between two horocycles of Euclidean radii $\rho_1$ and $\rho_2$  
	tangent to the unit circle at the same point. Then the spectrum of 
	$T_U$ is the interval 
\[ \left[0, \alpha^{- \frac{1}{\alpha -1}} - \alpha^{- \frac{\alpha}{\alpha-1}}\right]\]
	where
\begin{equation}\label{eq-alpha}
\alpha = \frac{\frac{1}{\rho_1}-1}{\frac{1}{\rho_2}-1}.
\end{equation}
\end{enumerate}
\end{theorem}

Consequently, the spectrum of $T_U$ consists only of essential points. 

 Another situation with $\Omega=\disk$ in which one can compute the 
 spectrum is the following. In hyperbolic  plane geometry, given a 
 hyperbolic geodesic $\Gamma$, a \textit{hypercycle}  or a 
 \textit{equidistant-curve} with \textit{axis} $\Gamma$ is a curve 
 consisting of points at a fixed signed distance from $\Gamma$ 
 (see \cite[Chapter XI]{coxeter}). In the Poincar\'e disc model, the 
 geodesic $\Gamma$ becomes a Euclidean circle or straight line which 
 meets the bounding unit circle $\partial \disk$ orthogonally, whereas a 
 hypercycle becomes an arc of a Euclidean circle or a segment of a straight 
 line which passes through the two points in $\overline{\Gamma}\cap \partial \disk$. 
 Hypercycles with the same axis may be referred to as \textit{coaxial}. 
 The region between two coaxial hypercycles will be referred to as a 
 \textit{hypercyclic lune}. The shaded region in Figure \ref{fig-hypercycle} 
 shows the hypercyclic lune between two coaxial hypercycles with axis 
 the diameter $(-1,1)$ of the unit circle, which is clearly a hyperbolic geodesic. 
 
 The  ideal endpoints  $\overline{\Gamma}\cap \partial \disk$ of the 
 geodesic $\Gamma$ divide the circle $\partial \disk$ into two arcs 
 $C_1, C_2$. The region bounded by a hypercycle with axis $\Gamma$ 
 and one of the arcs $C_j$  will be called a \textit{hypercyclic crescent}.  
 The complement of the shaded hypercyclic lune of Figure \ref{fig-hypercycle} 
 is the disjoint union of two hypercyclic crescents.
 
\begin{theorem}\label{thm-hypercycle}
		Let $\disk$ be the unit disc in $\cx$.
\begin{enumerate}
\item If $U\subset \mathbb{D}$ is a hypercyclic crescent, then the 
	spectrum of $T_U$ is the full interval $[0,1]$. 
\item If $U\subset \mathbb{D} $ is a hypercyclic lune determined 
	by two coaxial hypercycles,	then the spectrum of $T_U$ is an 
	interval of the form $[0,  \norm{T_U}]$ where $0<\norm{T_U}<1$. 
\end{enumerate}
\end{theorem}
Observe that  there is a one parameter group of biholomorphic automorphisms 
of the unit disc which maps a horocyclic strip  (or a horodisc) to itself. Similarly, 
there is a one-parameter group of automorphisms of the unit disc which maps 
a hypercyclic lune to itself. Both assertions can be seen by mapping the unit 
disc onto the upper half plane. The presence of such a large group of 
automorphisms implies noncompactness of the restriction operator.

\begin{figure}
	\begin{center}
		\begin{tikzpicture}[scale=2.5]x
		\draw (0,0) circle (1cm); 
		\draw [fill=black!40](1,0) arc [radius=1.11803, 
		start angle=26.56505, end angle= 153.43494];
		\draw [fill=white](1,0) arc [radius=1.25, 
		start angle=36.86989, end angle=143.130102];
		\end{tikzpicture}
	\end{center}
	\caption{Region between two hypercycles}\label{fig-hypercycle}
\end{figure}
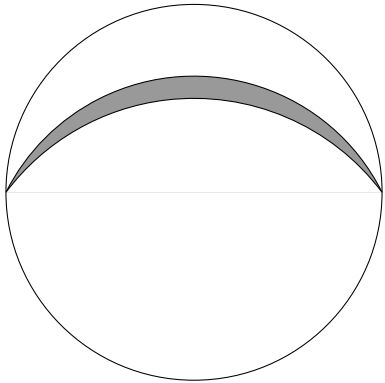

\begin{theorem}\label{thm-symmetry} 
	Let $\Omega$ be a smooth bounded pseudoconvex domain in $\cx^n$ 
	and  $U\subset \Omega$ be a non-empty open subset. Assume that there 
	exists a closed noncompact subgroup $G$ of the  topological group  
	$\text{Aut}(\Omega)$ of biholomorphic automorphisms of $\Omega$ 
	such that $f(U)\subset U$ for all $f\in G$. Then the operator $T_U$ is non-compact.
\end{theorem}
We note that the group of automorphisms $\text{Aut}(\Omega)$ is naturally a 
topological group in the compact-open topology, and when $\Omega$ is bounded, 
this group has a natural Lie group structure. See \cite{NarasimhanBook} for more details.

Theorem \ref{thm-symmetry} is purely qualitative, and only gives the existence of 
an essential spectrum of $T_U$ without describing its nature. We recall that 
the \textit{essential spectrum}  is part of the spectrum that does not contain the 
eigenvalues with finite multiplicity (see, for example, \cite[pg 73]{DaviesBook}).
In the absence of a large group of automorphisms of $\Omega$ which maps 
$U$ to itself, one cannot hope for a complete determination of the spectrum of 
$U$ relative to $\Omega$ along the lines of Theorems \ref{thm-horostrip} 
or \ref{thm-hypercycle}. However, considering that $T_U$ is a Toeplitz 
operator on $A^2(\Omega)$, one can obtain further interesting information 
about the spectrum. 

Theorems \ref{thm-horostrip} and  \ref{thm-hypercycle} may give the impression that
regions in the disc which have cusps or corners
 on the boundary typically have noncompact restriction 
operators. A counterexample to this is shown in Figure \ref{fig-ideal}, which shows an 
``ideal triangle''  $\Delta$ in the Poincaré disc. This non-compact open set is bounded 
by three entire hyperbolic geodesics, and it is known from hyperbolic geometry that the 
hyperbolic area $A_H(\Delta)= \pi$, provided the metric is normalized to have 
Gaussian curvature $-1$. In fact, in this case, not only are the operators
$T_U$ and $R_U$ compact, but $T_U$ belongs to the \textit{trace class}, i.e., 
the \textit{Schatten 1-norm} of $T_U$ 
\[ \norm{T_U}_{S_1} = \sum_{j=1}^\infty \lambda_j <\infty\]
where $\{\lambda_j\}$ are the eigenvalues of the compact operator $T_U$. 
In fact, for a domain $\Omega\subset \cx^n$ and an open set $U\subset \Omega$, 
the Schatten-1 norm of $T_U$ (i.e. trace of $T_U$) may be computed by the trace 
formula
\begin{equation}\label{eq-trace}
\norm{T_U}_{S_1} = \int_U B_\Omega(\zeta, \zeta)dV(\zeta),
\end{equation}
where $B_\Omega$ is the Bergman kernel of $\Omega$ (see 
Section \ref{sec-schatten} for a proof). When $\Omega$ is homogeneous 
(for example the unit disc or ball) this is the same (up to a multiplicative 
constant) as the Riemannian volume of $U$ in the Bergman metric. 
Therefore, it follows that for homogenous domains $\Omega$, the volume 
of $U$ is a a spectral invariant.  For the unit disc, we have the normalized 
volume element $dA_H(z) = \frac{4dV(z)}{(1-\abs{z}^2)^2}$. So  it follows 
that the trace of the operator $T_\Delta$  on $A^2(\disk)$ is given by
\[\norm{T_\Delta}_{S_1}
= \int_\Delta B(z,z)dV(z)
= \int_\Delta \frac{1}{4\pi} dA_H(z)
= \frac{1}{4},\]
so that the restriction operator $R_U$ belongs to the Schatten $\frac{1}{2}$-class, and 
\[ \norm{R_\Delta}_{S_{1/2}}=\sqrt{\norm{T_\Delta}_{S_1}}= \frac{1}{2}.\]
In particular, we have the crude norm estimate
\[ \norm{R_\Delta} < \frac{1}{4}.\]

More generally, we can measure the ``degree of compactness'' of the operators 
$R_U$ and $T_U$  using the \textit{Schatten $p$-norm} $\norm{T_U}_{S_p}$ 
 and the fact that $\norm{R_U}^2_{S_p}=\norm{T_U}_{S_{2p}}$  
(see \cite{zhu_book_2007}).  If $T_U$ is compact, since it is known 
to be  positive and self-adjoint, we have 
\[\norm{T_U}^p_{S_p} =\sum_{j=1}^\infty \lambda_j^p,\] 
where $\{\lambda_j\}_{j=1}^\infty$ are the eigenvalues of $T_U$. 
For each $p$, the Schatten norm $\norm{T_U}_{S_p}$ gives 
us a way of measuring the size of $U$ relative to $\Omega$.  
In  view of \eqref{eq-trace} we can say that the Schatten norms represent
``volume-like'' invariants of the biholomorphic geometry of subdomains. 
In Section \ref{sec-schatten}, we consider the problem of computing Schatten 
norms of $R_U$ and $T_U$.

\begin{figure}
\begin{center}
\begin{tikzpicture}[scale=2.5]
\draw [fill=black!40](0,0) circle (1cm); 
\draw [fill=white](0.866025,-0.5) arc [radius=1.7320508, start angle=60, end angle= 120];
\draw [fill=white](-0.866025,-0.5) arc[ radius=1, start angle = 210,  end angle=330];
\draw [fill=white](-0.866025,-0.5) arc [radius=1.7320508, start angle=-60, end angle=0];
\draw [fill=white](0,1) arc[ radius=1, start angle =90,  end angle=210];
\draw [fill=white](0,1) arc [radius=1.7320508, start angle=180, end angle= 240];
\draw [fill=white](0.866025,-0.5) arc[ radius=1, start angle =-30,  end angle=90];
\end{tikzpicture}
\end{center}
\caption{Ideal triangle}\label{fig-ideal}
\end{figure}
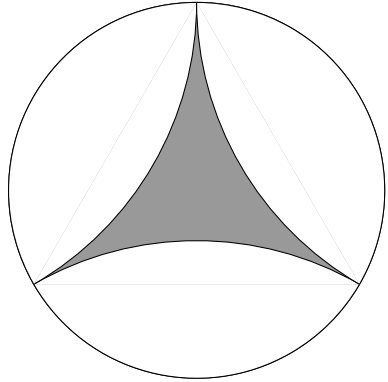

\subsection{Historical note} As early as 1927,  Stefan Bergman had developed the theory of 
``Doubly Orthogonal'' sequences of functions in Bergman spaces, which implicitly involves 
the restriction operator (see \cite{BergmanBook}.)  In \cite{Shapiro79} this relationship was 
made precise (this is basically our Proposition~\ref{PropCompact}).

When $U$ is a relatively compact subset  of 
$\Omega$, compactness of  $R_U$  follows immediately from  
Montel's  Theorem. In \cite[page 61ff.]{gunning66} this  was used in a proof of 
the  finite dimensionality of cohomology groups of compact Riemann surfaces. 
Henkin and Mityagin in \cite[Lemma 1.6]{henkinmityagin} further showed that  
in this case the operator $R_U$  is in each Schatten class. In \cite{lark}, the spectrum 
of the operator $T_U$ was studied when the ambient domain $\Omega$ is the unit 
disc and $U$ is a disjoint union of annuli, all centered at 0. Related results were 
also studied in \cite{luecking81, luecking84}.
	 
The non-compactness of the restriction operator when $\Omega$ 
is a convex domain in $\cx^n$ and $U$ is a dilated subdomain of $\D$ 
	(with a common boundary point with $\D$)
was used by Fu and Straube in \cite{FuStraube98} to study noncompactness 
of the $\dbar$-Neumann problem. These investigations were continued by 
Mijoung  Kim in her thesis \cite{KimThesis} and in \cite{Kim04}, where $U$ 
is said to be ``fat'' or ``thick'' in $\Omega$ if $R_U$ is not compact. Further 
extensions of these ideas may be found in 	\cite{SahutogluStraube06, DallAra18}.  

In the work of Vasilevski (\cite{vasilevski} and references therein), special cases 
of the operators $T_U$ on the Bergman space of the disc
arise in a different context. When the set $U$ has certain symmetries, it 
follows that the operator $T_U$ belongs to a commutative subalgebra of 
Toeplitz operators. We use results from \cite{vasilevski} to compute the 
spectra of certain subsets of the unit disc below 
(see Section \ref{sec-horosymmetry}).

\subsection{Acknowledgment} 
We thank Siqi Fu, Trieu Le,  and Laszlo Lempert for helpful comments 
and  discussions.
 
 \section{Basic facts}\label{sec-basic}
 
 \subsection{Definitions and notation}
	For a domain $\Omega$ in $\cx^n$,  recall that we denote by 
	$A^2(\Omega)$ the Bergman space of  holomorphic functions 
	$f$ on $\Omega$ such that 
	\[ \norm{f}_\Omega^2=\int_\Omega \abs{f}^2 dV <\infty,\]
	where $dV$ is the natural volume form of $\cx^n$ (identified 
	with $\rl^{2n}$). Then $A^2(\Omega)$   is a Hilbert space with 
	the norm $\norm{\cdot}_\Omega$.  
	
 Since $R_U:A^2(\Omega)\to A^2(U)$  is an injective bounded linear 
map between Hilbert spaces, the adjoint $R_U^*$ is a bounded linear 
map from $A^2(U)$ to $A^2(\Omega)$, which has dense range.  
Let $B_\Omega:\Omega\times\Omega\to \cx$ denote the Bergman 
kernel of $\Omega$. Recall that the operator $T_U$ on $A^2(\Omega)$ 
is defined to be the composition $R_U^* R_U$. 

\begin{proposition} \label{prop-t} 
Let $\Omega$ be a domain in $\cx^n$ and $U$ be a non-empty open 
subset of $\Omega$. Then for $g\in A^2(\Omega)$,  we have
\[ T_U g= P(\chi_U g),\]
where $P:L^2(\Omega)\to A^2(\Omega)$ is the Bergman projection 
and $\chi_U$ is the characteristic function of the set $U$. Furthermore, 
the operator $T_U$ is injective and its range is dense in $A^2(\Omega)$.
 \end{proposition}
 
 \begin{proof}
 Let $f\in A^2(U)$. We will first show that the adjoint operator   
 $R_U^*:A^2(U)\to A^2(\Omega)$ is given by the formula 
 \begin{equation}\label{eq-rustar}
 R_U^*f(z)= \int_U f(w)B_\Omega(z,w)dV(w).
\end{equation}
Notice that this may be  written in the form $R_U^* f =P(\wt{f})$, where 
$\wt{f}\in L^2(\Omega)$ is the function which coincides with $f$ on $U$ 
and is zero on $\Omega\setminus U$. This shows that the right hand side 
of  \eqref{eq-rustar} defines a bounded linear map from $A^2(U)$ 
to $A^2(\Omega)$.	To prove \eqref{eq-rustar}, let $g\in A^2(\Omega)$. Then 
 \begin{align*} \ipr{R_U^*f,g}_{A^2(\Omega)}= \ipr{f,R_U g}_{A^2(U)}
 = \ipr{\wt{f},g}_{L^2(\Omega)}
 =\ipr{P(\wt{f}), g}_{A^2(\Omega)}.
 \end{align*}
We note that the last equality follows on decomposing $\wt{f}$ into its 
components along $A^2(\Omega)$ and  $A^2(\Omega)^\perp$. This 
shows that $R_U^*f-P(\wt{f})\in A^2(\Omega)$ is orthogonal to each element 
of $A^2(\Omega)$. Hence, $R_U^*f-P(\wt{f})=0$ which establishes \eqref{eq-rustar}. 

For $g\in A^2(\Omega)$ and $z\in \Omega$ we now have
 \begin{align*}
 T_Ug(z)&=R^*_U\left(R_U g  \right)(z)\nonumber\\
 &= \int_U g(w) B_\Omega(z,w) dV(w)\\
&=\int_\Omega\left(\chi_U(w)g(w) \right)B_\Omega(z,w)dV(w) \nonumber\\
&=P(\chi_U g)(z)\nonumber
 \end{align*}
 Finally we prove that $T_U$ has dense range. Connectedness of $\Omega$ together 
 with the identity principle imply that the operator $R_U$ is  injective. Then we conclude 
 that $T_U$ is injective because if $T_Uf=0$ for some $f\in A^2(\Omega)$, then 
 $\ipr{T_Uf, f}_\Omega= \norm{R_Uf}^2=0$ which, in turn, implies that $f=0$.  Finally, 
 since $T_U$ is injective, it follows that  its adjoint $T_U=(T_U)^*$ has dense range.
 \end{proof}

Proposition \ref{prop-t} shows that  $T_U$ is the \textit{Toeplitz operator } 
on the Bergman space $A^2(\Omega)$ with  symbol $\chi_U$  
(see \cite{zhu_book_2007}). It is clear that $T_U$ is strictly positive. 
 Note that  the operators $R_U, R_U^*, T_U$  are simultaneously 
 compact or non-compact (cf. \cite[Theorems 1.16, 1.17]{zhu_book_2007}). 
 Similarly, the operators  $R_U, R_U^*, T_U$  simultaneously have closed 
 range or do not have closed range. Furthermore, 
 \[ \norm{R_U}=\norm{R^*_U}= \sqrt{\norm{T_U}}.\]

 \subsection{Biholomorphic invariance}
 Let $\phi:\wt{\Omega}\to \Omega$ be a biholomorphic map. 
We can associate an isometric isomorphism of Bergman spaces
\[ \phi^\sharp: A^2(\Omega)\to A^2(\wt{\Omega})\]
by defining for $f\in A^2(\Omega)$ and $w\in \wt{\Omega}$
\begin{equation}\label{eq-inducedmap}
 (\phi^\sharp f)(w)= \det \phi'(w)\cdot  f(\phi(w))
\end{equation}
where $\phi'(w):\cx^n\to \cx^n$ is the $\cx$-linear map which is the 
complex derivative of $\phi$ at $w$.  It follows from the change 
of variables formula that $\phi^\sharp$ is an isometric 
isomorphism of Hilbert spaces. Consequently, we have
\[ (\phi^\sharp)^*\cdot \phi^\sharp = {\rm id}_{A^2(\Omega)}.\]

\begin{proposition}\label{prop-invariance} 
	Let  $\phi:\wt{\Omega}\to \Omega$ be a biholomorphic map, and let $\wt{U}$ 
	is an open subset of $\wt{\Omega}$ such that $\phi(\wt{U})=U$. Then  we have
	\begin{equation}\label{eq-tuinvariance}
	T_{\wt{U}}= \phi^\sharp \circ T_U \circ (\phi^\sharp)^{-1}.
\end{equation}
In particular, the self-adjoint operators $T_U$ and $T_{\wt{U}}$ are isospectral.
\end{proposition}

\begin{proof}
 It is clear that the following diagram commutes.
\[
\begin{CD}
A^2(\Omega) @>\phi^\sharp>> A^2(\wt{\Omega})\\
@VVR_UV @VVR_{\wt{U}}V\\
A^2(U) @>\phi^\sharp_U>> A^2(\wt{U})
\end{CD}
\]
where $(\phi^\sharp_U f)(w)=\det \phi'(w)\cdot  f(\phi(w))$ for 
$f\in A^2(U)$ and $w\in \wt{U}$. That is, $\phi^\sharp_U=(\phi|_U)^\sharp$. 
It follows that
$R_{\widetilde{U}}= \phi_U^\sharp \circ R_U\circ (\phi^\sharp)^{-1}  $.
Then 
\begin{align*}
T_{\wt{U}}&=\left( \phi^\sharp_U \circ R_U \circ (\phi^\sharp)^{-1}\right)^* 
\circ \left(\phi^\sharp_U  \circ R_U\circ (\phi^\sharp)^{-1}\right)\\
&=((\phi^\sharp)^{-1})^*\circ R_U^*
\circ(\phi^\sharp_U)^*\circ  \phi^\sharp_U
\circ R_U\circ (\phi^\sharp)^{-1}\\
&= ((\phi^\sharp)^{-1})^*\circ R_U^*\circ R_U\circ (\phi^\sharp)^{-1}\\
&= \phi^\sharp \circ T_U \circ (\phi^\sharp)^{-1}.
\end{align*}
Hence, the proof of the proposition is complete.
\end{proof}

Note that for any measurable subset $U\subset\Omega$, the formula
 \[ T_U =P\circ m_{\chi_U}\]
 (where $P:L^2(\Omega)\to A^2(\Omega)$ is the Bergman projection and 
 $m_{\chi_U}:A^2(\Omega)\to L^2(\Omega)$ is the multiplication operator induced 
 by $\chi_U$)
 defines a positive-definite self-adjoint Toeplitz operator satisfying the invariance property 
 of Proposition \ref{prop-invariance}. Therefore, this allows us to extend the notion of 
spectrum of a subset  from open subsets of $\Omega$ to arbitrary measurable subsets of 
 $\Omega$. However,  since we are motivated by the relation of $T_U$ with the restriction 
 operator $R_U$, we will continue to assume that $U$ is open.

 \subsection{Intrinsic Bergman spaces}\label{SubsectionManifold}
This invariance property is a manifestation of the fact that it is possible 
to define intrinsic versions of the Bergman space and the operators 
$R_U, R_U^*,  T_U$ without any reference to a given volume form $dV$.  

Let $\Omega$ be a complex manifold of dimension $n$. We denote 
the \textit{intrinsic Bergman space} of $\Omega$ by $A^2_{n,0}(\Omega)$. 
This is, by definition the Hilbert space, of those  holomorphic $(n,0)$-forms 
$f$ on $\Omega$ such that
\[ \norm{f}^2_\Omega = (\sqrt{-1})^{n^2}\int_\Omega f \wedge \ol{f} <\infty.\]
If $\Omega$ is a domain in $\cx^n$, the Hilbert space is clearly 
$ A^2_{n,0}(\Omega)$ is   isometrically  isomorphic to $A^2(\Omega)$ 
in a natural way.  If $\iota_U:U\to \Omega$ is the inclusion map, the 
pullback operator $\iota_U^*: A^2_{n,0}(\Omega)\to A^2_{n,0}(U)$  
is defined invariantly, i.e. without any reference to a distinguished 
metric or volume form on $\Omega$. The identification of 
$A^2(\Omega)$ with $A^2_{n,0}(\Omega)$ given above identifies 
the operator $R_U$ with the  invariantly defined operator $\iota_U^*$.  
Similarly, one can give invariant descriptions of $R_U^*, S_U$ and 
$T_U$.	Note that these invariant descriptions immediately show 
that the analytic properties of these operators are invariant under 
biholomorphisms, as seen in Proposition \ref{prop-invariance}.

\section{Compact restriction operators} \label{sec-spectral}

\subsection{Eigenvalues and eigenvectors} \label{sec-purepoint}
In this section we explore some spectral properties of $T_U$ when
there are eigenvalues in the spectrum.

\begin{proposition} \label{cor-eigenvalue}
Let $\Omega$ be a domain in $\cx^n$ with nontrivial 
Bergman space and $U$ be a non-empty {}{open subset} 
of $\D$ such that $\D\setminus \overline{U}\neq \emptyset$.
Suppose that $\lambda$ is an eigenvalue of $T_U$. Then
 \begin{equation} \label{eq-lambdajbounds}
 0<\lambda <1.
 \end{equation}
Furthermore, whenever 0 or 1 is in the spectrum of   $T_U$,    
it belongs to the essential spectrum of $T_U$.
\end{proposition}
\begin{proof} 
Let $\lambda$ be the eigenvalue of $T_U$ and $\phi$ is a 
 corresponding  eigenfunction. Then
\[
 \ipr{\phi,\phi}_U=\ipr{R_U\phi,R_U\phi}_U
=\ipr{R_U^*R_U\phi,\phi}_\Omega
= \ipr{T_U \phi,\phi}
= \lambda\ipr{\phi,\phi}_\Omega.
\]
Then $\lambda = \frac{\norm{\phi}_U^2}{\norm{\phi}_\Omega^2}$.
Using the facts that  $\D\setminus \overline{U}\neq \emptyset$, and 
that $\phi$ is a nonzero holomorphic function, we see that 
	\[ 0< \ipr{\phi,\phi}_U < \ipr{\phi, \phi}_\Omega.\]
It follows that $0<\lambda <1$. So, if 0 or 1 is in the spectrum 
of $T_U$ then it is not an eigenvalue  and, hence, it belongs to 
the essential spectrum of $T_U$.
\end{proof}

We now suppose that $T_U$ has  \textit{pure point spectrum}, i.e., there is 
an orthonormal basis of $A^2(\Omega)$ consisting of eigenvectors of 
$T_U$. If $\{\phi_j\}$ is an orthonormal basis of $A^2(\Omega)$ 
consisting of eigenvectors of $T_U$ and $\lambda_j$ is the 
eigenvalue corresponding to $\phi_j$, then it is clear that 
\[ {\rm Spec}\,T_U = \ol{\{\lambda_j\}_{j=1}^\infty}, \]
the closure of the set of eigenvalues. Furthermore, for each 
$f\in A^2(\Omega)$ we can write,
\begin{equation}\label{eq-tuspectral}
T_Uf=\sum_{j=0}^{\infty}\lambda_j \langle f,\phi_j\rangle \phi_j,
\end{equation}
where  the series converges in the norm of $A^2(\Omega)$.  

In order not to repeat many times  the conclusion  (1)  of 
Proposition \ref{PropCompact}, let us say that 
an orthonormal basis $\{\phi_j\}$ of $A^2(\Omega)$ is  
{\textit{adapted}} to an open subset $U\subset \Omega$, 
if $\{\phi_j|_U\}$ is an orthogonal set in $A^2(U)$, i.e.,  if $j\not=k$, then 
\[ \ipr{\phi_j,\phi_k}_U=\int_{U} \phi_j(z)\ol{\phi_k(z)} dV(z)=0.\]
In the literature, one usually expresses this by saying that the family $\{\phi_j\}$ 
forms a {\em doubly orthogonal sequence of functions} (with respect to the two 
domains $\Omega$ and $U$). Doubly orthogonal systems were  discovered and 
named by Stefan Bergman as far back as 1927  (see 	\cite{BergmanBook}). 
For some recent results related to doubly orthogonal systems see 
\cite{Andersson00,AizenbergAytunaDjakov01,GustafssonPutinarShapiro03,PutinarPutinar06}.

\begin{proposition}\label{prop-adapted}

Let $\Omega$ be a domain in $\cx^n$ with nontrivial Bergman space, 
$U$ be a non-empty open subset of $\D$,  and $\{\phi_j\}$ be an orthonormal 
basis of $A^2(\Omega)$. Then the following are equivalent.
\begin{enumerate}
\item $\{\phi_j\}$  is adapted to $U$,
\item  $T_U$ has pure point spectrum and each $\phi_j$ is an eigenvector of $T_U$. 
\end{enumerate}
If either condition (1) or (2) (and therefore both) are satisfied, the eigenvalue 
$\lambda_j$ of  $T_U$ corresponding to the  eigenvector  $\phi_j$  is given by
	\begin{equation}\label{eq-lambdaj}
	 \lambda_j= \ipr{\phi_j,\phi_j}_U.
	 \end{equation}
\end{proposition}
\begin{proof}
First we will prove that (1) implies (2). Assume that 
$\{\phi_j\}$ is {adapted} to $U$. 
Then  for any $j,k$ we have 
\[ \ipr{T_U\phi_j, \phi_k}_\Omega
=  \ipr{R^*_UR_U\phi_j, \phi_k}_\Omega
= \ipr{R_U\phi_j,R_U\phi_k}_\Omega
= \ipr{\phi_j,\phi_k}_U= \ipr{\phi_j,\phi_j}_U\delta_{jk}.\]
Since $\{\phi_j\}$ is an orthonormal basis, this means that 
$T_U\phi_j= \ipr{\phi_j,\phi_j}_U\phi_j$ for each $j$ .
It follows that $\phi_j$ is an eigenvector of $T_U$ with eigenvalue 
$\lambda_j=\ipr{\phi_j,\phi_j}_U$. Since $\{\phi_j\}$ is an orthonormal 
basis of $A^2(\Omega)$ it follows that $T_U$ has pure point spectrum.

To prove the converse, let us assume (2), and so $T_U\phi_j=\lambda_j \phi_j$, 
where $\lambda_j$ is the eigenvalue of $T_U$ corresponding to $\phi_j$.   
Then we have
\begin{equation}\label{Eqn1}
 \ipr{\phi_j,\phi_k}_U=\ipr{R_U\phi_j,R_U\phi_k}_U
=\ipr{R_U^*R_U\phi_j,\phi_k}_\Omega
= \ipr{T_U \phi_j,\phi_k}
= \lambda_j\ipr{\phi_j,\phi_k}_\Omega
=\lambda_j\delta_{jk}
\end{equation}
for all $j$ and $k$. Therefore, the family $\{\phi_j|_U\}$ 
is orthogonal in $A^2(U)$, and $\lambda_j= \ipr{\phi_j,\phi_j}_U$.
\end{proof}

The next proposition,  first proved by Bergman (see \cite{BergmanBook})
was proved by Shapiro (\cite{Shapiro79}) by an   application 
of the spectral theorem for compact self-adjoint operators, 
and describes  the consequences of compactness 
of  $R_U$ (or  equivalently, that of  $T_U= R_U^* R_U$). 

\begin{proposition}[Bergman-Shapiro]\label{PropCompact}
	Let $\Omega$ be a domain in $\cx^n$ with nontrivial Bergman space and 
	$U$ be a non-empty {}{open subset} of $\Omega$. Suppose that the restriction 
	operator $R_U:A^2(\Omega)\to A^2(U)$ is compact. Then
	\begin{enumerate}
		\item there is an orthonormal basis $\{\phi_j\}$ of 
		$A^2(\Omega)$ such that the restrictions $R_U(\phi_j)=\phi_j|_U$  
		form an orthogonal set in $A^2(U)$;
		\item the operator $T_U$ has  the following spectral representation as a series 
		converging in $A^2(\Omega)$. 
		\[T_U f 
		= \sum_{j=1}^\infty \lambda_j \ipr{f,\phi_j}_\Omega \phi_j,\quad (f\in A^2(\Omega))\] 
		where the eigenvalues are given by
		\[ \lambda_j=\ipr{\phi_j,\phi_j}_{U}= \int_U \abs{\phi_j}^2dV.\]
	\end{enumerate}
\end{proposition}
Here, $\ipr{\cdot,\cdot}_\Omega$ (resp. $\ipr{\cdot,\cdot}_U$) denoted 
the inner product of $A^2(\Omega)$ (resp. $A^2(U)$.) We apply 
Proposition \ref{PropCompact} to compute some spectra in highly 
symmetric situations in Examples \ref{ex-reinhardt} and \ref{ex-dilation},
and in Proposition \ref{prop-compactdisc}.

We provide a proof of Proposition \ref{PropCompact} here 
for the convenience of the reader.
\begin{proof}[Proof of Proposition \ref{PropCompact}]
First we will prove (1). Since  $R_U$ is compact, then $T_U=R_U^*R_U$ is a 
compact self-adjoint operator, and by the spectral theorem, 
there is an orthonormal basis $\{\phi_j\}$  of  $A^2(\Omega)$ 
consisting of eigenvectors of $T_U$.  The implication (2) $\Rightarrow $ (1) 
of Proposition \ref{prop-adapted} immediately gives us the result.

Assertion (2) now follows on noting that in the spectral representation 
\eqref{eq-tuspectral}, when the orthonormal basis $\{\phi_j\}$ is {adapted} 
to $U$, the eigenvalues are given by \eqref{eq-lambdaj}.
\end{proof}

\subsection{Some examples}\label{sec-example}
\begin{example}\label{ex-reinhardt}
Let $\mathbb{Z}_+=\{0,1,2,3,\ldots\}$ and suppose that $\Omega$ 
is a complete Reinhardt domain in $\cx^n$ and $U\subset \Omega$ 
is a Reinhardt subdomain. Then for each multi-index 
$\alpha\in \mathbb{Z}_+^n$, we consider the function  
$\phi_\alpha\in A^2(\Omega)$ given by
\[\phi_\alpha(z) = \frac{z^\alpha}{\norm{z^\alpha}_\Omega}.\]
Then $\{\phi_\alpha\}$ is an orthonormal basis of $A^2(\Omega)$. 
Further, since $U$ is also Reinhardt, it follows that 
$\ipr{\phi_\alpha,\phi_\beta}_U=0$ if $\alpha\not=\beta$. 
Hence the orthonormal basis $\{\phi_\alpha\}$ is {adapted} 
to $U$ and the eigenvalues of $T_U$ are given by the ratios 
\begin{equation}\label{eq-eigenvaluealpha}
\lambda_\alpha
= \frac{\norm{z^\alpha}_U^2}{\norm{z^\alpha}_\Omega^2}.
\end{equation}
Therefore, we have 
\begin{equation}\label{eq-reinhardtnorm}
\norm{R_U} 
=\sqrt{\norm{T_U}} 
= \sup_{\alpha\in \mathbb{Z}_+^n}
 \frac{\norm{z^\alpha}_U}{\norm{z^\alpha}_\Omega}.
\end{equation}
 \end{example}
 \begin{example}\label{ex-dilation}
Let us now specialize to the case when $U$ is a dilated version of the 
complete Reinhardt domain $\Omega$, i.e., there is a $0<\rho<1$ 
such that $U=\rho\Omega=\{\rho z| z\in \Omega\}$. Then for each 
multi-index $\alpha \in \mathbb{Z}_+^n$, we have 
\begin{align*}
\norm{z^\alpha}_U^2 &= \int_U \abs{w^\alpha}^2 dV(w)\\
&= \int_\Omega \abs{(\rho z)^\alpha}^2 \rho^{2n}dV(z)\\
&= \rho^{2\abs{\alpha}+2n} \int_\Omega \abs{ z^\alpha}^2 dV(z)\\
&= \rho^{2(\abs{\alpha}+n)} \norm{z^\alpha}_\Omega^2.
\end{align*}
Consequently, the eigenvalues of $T_U$ are
\begin{equation}\label{EqnEV}
\lambda_\alpha = \rho^{2(\abs{\alpha}+n)}, \alpha \in \mathbb{Z}_+^n.
\end{equation}
Therefore, the norm of the restriction operator is given by 
\begin{equation}\label{eq-normru}
\norm{T_U}=\norm{R_U}^2 = \sup_\alpha \lambda_\alpha = \rho^{2n}.
\end{equation}
\end{example}

We can now recapture a classic fact about annuli.
 
\begin{proposition} 
For $0<r_2<r_1$, let $A(r_1,r_2)= \{z\in \cx:r_2<\abs{z}<r_1\}$ be an annulus 
in the plane.  If $A(r_1', r_2')$ is conformally equivalent to $A(r_1, r_2)$, then 
\[\frac{ r_2}{r_1} = \frac{ r_2'}{r_1'}.\]
\end{proposition}
\begin{proof}
Let $f:A(r_1,r_2)\to A(r_1', r_2')$ be a conformal map, and let us set 
$\Omega= \{z\in \cx:\abs{z}<r_1\}, U= \{z\in \cx:\abs{z}<r_2\}$, 
$\wt{\Omega}= \{z\in \cx:\abs{z}<r_1'\}, \wt{U}= \{z\in \cx:\abs{z}<r_2'\}$. 
Notice that $U=\left(\frac{r_2}{r_1}\right) \Omega$ and 
$\wt{U}=\left(\frac{r_2'}{r_1'}\right) \wt{\Omega}$, so that  thanks 
to \eqref{eq-normru}, we have $\norm{T_U}=\left(\frac{r_2}{r_1}\right) ^2$ 
and $\norm{T_{\wt{U}}}= \left(\frac{r_2'}{r_1'}\right)^2$.

Applying an inversion if necessary, we can assume that as 
$|f(z)|\to r_1'$ as $\abs{z}\to r_1$. We can now apply repeated Schwarz 
reflection in the inner circles, followed by an appeal to the Riemann 
removable singularity theorem to extend the map $f$ to a biholomorphic 
map from $\Omega$ to $\wt{\Omega}$ such that $f(U)=\wt{U}$. 
It follows that $\norm{T_U}=\norm{T_{\wt{U}}}$, and therefore 
$r_2/r_1=r_2'/r_1'$.
\end{proof}

We now compute the spectrum of $T_U$ for an arbitrary relatively compact 
subdisc of the unit disc $\mathbb{D}$. 

\begin{proposition}\label{prop-compactdisc}
Let $z_0\in \cx, r>0,$ 
$A= \sqrt{\frac{(1+ {\abs{z_0}+r})(1- \abs{z_0}+r)}{(1-\abs{z_0}-r)(1+ \abs{z_0}-r)}},$ 
and $U=\{z\in \cx:|z-z_0|<r\}$ such that 
$\overline{U}\Subset \D=\mathbb{D}$. Then the spectrum of 
$T_U$ is composed of eigenvalues $\{\lambda_k\}_{k=0}^\infty$ where
\[ \lambda_k = \left(\frac{A-1}{A+1}\right)^{2k+2}.\]
\end{proposition} 	
	
\begin{proof}
Recall that the distance in the Poincaré hyperbolic metric  of the disc  
(suitably normalized) from the origin to the point $z\in \disk$ is given by
\[ \eta(0,z) = \log\frac{1+\abs{z}}{1-\abs{z}}.\] 
It follows that the disc $\{z\in \cx:\abs{z}<r\}$ is a hyperbolic disc centered at the 
origin and of radius $\eta(0,r)$. Since any two hyperbolic discs  of the same 
radius are congruent under ${\rm Aut}(\mathbb{D})$, it follows from 
\eqref{EqnEV}  that the spectrum of a hyperbolic disk of radius $\rho$ 
consists of the points $\{r^{2k+2}, k=0,1,\dots\}$, where $r>0$ is such 
that $\rho= \eta(0,r)$. A computation using the formula for $\eta$ shows that
\[ r = \frac{e^\rho-1}{e^\rho+1}.\]
Now consider the disc $U=\{z\in \cx:|z-z_0|<r\}$.  Let $P$ and $Q$ be the points 
at which the diameter through $z_0$ meets the circumference of $U$, so that 
$\abs{P}=\abs{z_0}+r$ and $\abs{Q}=\abs{\abs{z_0}-r}$ (if $z_0=0$, take any 
diameter). If $\abs{z_0}\geq r$, then $0\not \in U$, and the hyperbolic radius 
of $U$ (thought of as a hyperbolic disc) is given by
\begin{align*}
	 \frac{1}{2}\left( \eta(0,P)- \eta(0,Q)\right)
	 &=\frac{1}{2} \log\left(\frac{1+\abs{\abs{z_0}+r}}{1-\abs{\abs{z_0}+r}}
	 \cdot\frac{1- \abs{\abs{z_0}-r}}{1+\abs{\abs{z_0}-r}}  \right)\\
	 &=\frac{1}{2} \log\left(\frac{1+\abs{z_0}+r}{1-\abs{z_0}-r}\cdot
	 \frac{1- \abs{z_0}+r}{1+\abs{z_0}-r}  \right)\\
	   &= \log A.
 \end{align*}
Now, if $\abs{z_0}< r$, then $0\in U$, and in this case, the hyperbolic radius 
of $U$ is given by
\begin{align*} 
	\frac{1}{2}\left( \eta(0,P)+ \eta(0,Q)\right)
	&=\frac{1}{2}\log\left(\frac{1+\abs{\abs{z_0}+r}}{1-\abs{\abs{z_0}+r}}\cdot
	\frac{1+\abs{\abs{z_0}-r}}{1-\abs{\abs{z_0}-r}}  \right)\\
	 &=\frac{1}{2} \log\left(\frac{1+\abs{z_0}+r}{1-\abs{z_0}-r}\cdot
	 \frac{1 -\abs{z_0}+r}{1+ \abs{z_0}-r}  \right)\\ 
	&= \log A.
\end{align*}
Hence the proof of the proposition is complete. 
\end{proof}
	
\section{Complementary  domains and  proof of Theorem~\ref{thm-basic}}
\subsection{Spectral properties of complementary domains}
\begin{proposition} \label{prop-eigen}
	Let $\Omega$ be a domain in $\cx^n$ with nontrivial Bergman space and 
	$U$ be a non-empty {}{open subset} of $\D$ 
	such that $\D\setminus \overline{U}\neq \emptyset$. 
	Assume that $\partial U$ has zero Lebesgue measure. Then 
	\begin{enumerate}
	\item $T_U$ and $T_{\Omega\setminus \ol{U}}$ commute as 
	operators on $A^2(\Omega)$ and 
\begin{equation}\label{eq-spectrumrprime}
 \mathrm{ Spec}(T_{\Omega\setminus \ol{U}})
 =  \{1-\lambda| \lambda \in  \mathrm{ Spec}(T_{U})\}, 
\end{equation}
\item if there is an orthonormal basis $\{\phi_j\}$ of 
$A^2(\Omega)$ adapted to $U$, this basis is also adapted to 
$\Omega \setminus \ol{U}$.
\end{enumerate}
\end{proposition}

\begin{proof} 
	By Proposition~\ref{prop-t} we see that for each $f\in A^2(\Omega)$  
	we have since $\chi_U+ \chi_{\Omega \setminus \ol{U}}=1$ a.e. 
\begin{align*}
T_U f(z)+ T_{\Omega\setminus \ol{U}}f(z)
&= P(\chi_U f) + P(\chi_{\Omega \setminus \ol{U}}f)\\
&= P((\chi_U+ \chi_{\Omega \setminus \ol{U}})f)\\
&= P(f)\\
&= f.
\end{align*}
Then we have
\[ T_{\Omega\setminus\ol{U}}= {\rm id}_{A^2(\Omega)}-T_U.\]
 It follows that $T_U$ and $T_{\Omega\setminus \ol{U}}$ 
 commute. Consequently
$ \mathrm{ Spec}(T_{\Omega\setminus \ol{U}})= \mathrm{ Spec}(I- T_U)$ 
and  \eqref{eq-spectrumrprime} follows.

Assuming that $\{\phi_j\}$ is {adapted} to $U$, the fact 
that it is {adapted} 
to $\Omega\setminus \ol{U}$ follows on writing 
\[ \ipr{\phi_j,\phi_k}_\Omega=\ipr{\phi_j,\phi_k}_U
+ \ipr{\phi_j,\phi_k}_{\Omega\setminus \ol{U}}.\]
Hence the proof of the proposition is complete. 
\end{proof}

Next we prove Theorem \ref{thm-basic}.
\begin{proof}[Proof of Theorem \ref{thm-basic}]
First consider part (1). Assume that $R_U$ is compact. We observe 
that since $\norm{R_U}=\sqrt{\norm{T_U}}$ and $\norm{T_U}$ is 
the largest eigenvalue of $T_U$, which is strictly less than 1 
by \eqref{eq-lambdajbounds}. Hence, we conclude that $\norm{R_U}<1$. 
 
To prove that $\|R_{\Omega\setminus\ol{U}}\|_e=1$,  note that the 
operator $T_U$ is compact since $R_U$ is compact. Then 0 is an essential 
point of the spectrum of $T_U$, by  Proposition~\ref{cor-eigenvalue}. 
Let $\{\phi_j\}$ be an orthonormal basis of $A^2(\Omega)$ consisting 
of eigenvectors of $T_U$, and let $\lambda_j$ be the eigenvalue of 
$T_U$ corresponding to the eigenvector $\phi_j$. It now follows from 
Proposition \ref{prop-eigen}  that the operator $T_{\Omega\setminus \ol{U}}$ 
has pure point spectrum, and $\phi_j$ is an eigenvector  with corresponding 
eigenvalue $1-\lambda_j$.   It follows that 1 is an essential point of the 
spectrum of $T_{\Omega\setminus \ol{U}}$, so that 
$\|T_{\Omega\setminus \ol{U}}\|_e=1$. Therefore, by 
\cite[Lemma 1]{CuckovicSahutoglu18} it follows that 
$\|R_{\Omega\setminus \ol{U}}\|_e=1$, and consequently,  
we have  $\|R_{\Omega\setminus \ol{U}}\|=1$.

We prove part (2) of Theorem \ref{thm-basic} in a slightly stronger form 
below in Proposition \ref{prop-closedrange}.  
\end{proof}	

\begin{proposition} \label{prop-closedrange}
Let $\Omega$ be a domain in $\cx^n$ with nontrivial Bergman space 
and  $U$ be a non-empty {}{open subset} of $\D$ such that 
$\D\setminus \overline{U}\neq \emptyset$ and $\partial U$ 
has zero Lebesgue measure. Then the following are equivalent
\begin{enumerate}
\item  $\norm{R_{U}}<1$,
\item $R_{\Omega\setminus\ol{U}}$ has closed range,
\item $T_{\Omega\setminus \ol{U}}$ is a linear homeomorphism. 
\end{enumerate}
\end{proposition}

\begin{proof}
Assume (1), i.e.,  $\norm{R_U}<1$.  Now  $\norm{T_U}= \norm{R_U}^2$ 
is the maximum of the spectrum of $T_U$, so that 
$\mathrm{ Spec}(T_U)\subset [0, \norm{R_U}^2]$.  
From \eqref{eq-spectrumrprime} we see that 
$\mathrm{ Spec}(T_{\Omega\setminus\ol{U}})
\subset \left[1-\norm{R_U}^2, 1\right].$ 
Then, by the spectral theorem (see, for example, \cite[275f.]{riesz-nagy}) 
we have for  each $f\in A^2(\Omega)$ we have
\[ \ipr{T_{\Omega\setminus\ol{U}}f, f}_\Omega 
\geq (1-\norm{R_U}^2) \norm{f}_\Omega^2,\]
which implies that
\begin{equation}\label{eq-rprime}
\norm{R_{\Omega\setminus\ol{U}} f}_{\Omega\setminus \ol{U}} 
\geq \sqrt{1-\norm{R_U}^2} \norm{f}_\Omega.
\end{equation}
Hence, $R_{\Omega\setminus\ol{U}}$ has closed range, i.e. (2).

Assume (2) now.  Thanks to Proposition~\ref{prop-t}, we already know 
that $T_U$ is injective and has dense range, so it suffices to show that 
$T_{\Omega\setminus\ol{U}}$ has closed range.  Now since by hypothesis,
$R_{\Omega\setminus\ol{U}}$ has closed range and we know it is injective, 
it follows from the open mapping theorem that there is a a $C>0$ such that 
$\norm{R_{\Omega\setminus\ol{U}}f}\geq C \norm{f}$ for all $f\in A^2(\Omega)$. 
This is equivalent to the condition that 
\begin{equation}\label{eq-tlb}
\ipr{T_{\Omega\setminus\ol{U}} f, f}_\Omega
= \ipr{(R_{\Omega\setminus\ol{U}})^* R_{\Omega\setminus\ol{U}} f, f}_\Omega 
\geq C^2\norm{f}_\Omega^2.
\end{equation}
Since $T_{\Omega\setminus\ol{U}}$ is self-adjoint, this means that 
$\norm{T_{\Omega\setminus\ol{U}}f}\geq C \norm{f}$ for each 
$f\in A^2(\Omega)$, which shows that $T_{\Omega\setminus\ol{U}}$ 
has closed range, thus completing the proof of (3). 

Now suppose (3) holds,  so that   $T_{\Omega\setminus\ol{U}}$ has  
closed range, which is equivalent to \eqref{eq-tlb}. It follows that  
$\mathrm{ Spec}(T_{\Omega\setminus\ol{U}}) \subset [C,1]$. 
Using \eqref{eq-spectrumrprime} we see that 
$\mathrm{ Spec}(T_U)\subset [0, 1-C]$.  Consequently, 
\[ \norm{R_U}^2 = \norm{T_U} \leq 1-C<1.\]
Hence, the proof of Proposition  \ref{prop-closedrange} is complete.
 \end{proof}

\section{Norm estimates on tangent domains and proof of Theorem \ref{thm-balls}}
	\label{sec-norm}
\subsection{Planar situation}
In the case when the dimension $n=1$, the following result gives the norm 
of the restriction operator on tangent domains.

\begin{proposition} \label{prop-generalplanar} 
Let $\Omega$ be a  bounded domain in $\cx$ and $U$  
be a non-empty open subset of $\D$. Suppose that there is a 
$p\in \partial\Omega\cap \partial U$ near which $\partial\Omega$ 
and $\partial U$ are $C^1$-smooth. Then  $\norm{R_U}=1$.
\end{proposition}
Note that thanks to Proposition \ref{PropCompact} this means that $R_U$ is non-compact. 
\begin{proof} 
Without loss of generality we assume that the origin is a boundary point 
and the negative $x$-axis is the outward  normal at the origin. Then 
for $\ep>0$ small we can find $\delta>0$ so that 
\[U_{\ep,\delta} \subset U\cap \mathbb{D}_{\delta} 
\subset \D\cap \mathbb{D}_{\delta} \subset V_{\ep,\delta}.\]
where $\mathbb{D}_{\delta}=\{z\in \cx:|z|<\alpha\}$, 
\[U_{\ep,\delta} 
=\left\{re^{i\theta}:|\theta|<\frac{\pi}{2}-\ep, 0<r<\delta\right\}, 
\text{ and  }  V_{\ep,\delta}
=	 \left\{re^{i\theta}:|\theta|<\frac{\pi}{2}+\ep, 0<r<\delta\right\}.\]
Let us choose $f_j(z)=a_jz^{-\alpha_j}$ where $\alpha_j =1-2^{-j}$ and 
$a_j=\delta^{\alpha_j-1}\sqrt{\frac{2-2\alpha_j}{\pi+2\ep}}$. Then one can 
compute that $a_j\to 0$ and $ \|f_j\|^2_{V_{\ep,\delta}}=1$. 
The fact that $a_j\to 0$ implies that $f_j\to 0$ on any compact set away 
from the origin. That is, the mass of $f_j$ accumulates near the origin as 
$j\to\infty$.  Furthermore, one can compute that 
$\|f_j\|^2_{U_{\ep,\delta}}=\frac{\pi-2\ep}{\pi+2\ep}$. 
Hence for every $\ep>0$ there exists $j$ such that 
\[\frac{\pi-2\ep}{\pi+2\ep} \leq \|f_j\|^2_{U}\leq \|f_j\|^2_{\D}\leq 1+\ep.\]
Therefore, $\|R_U\|=1$.	 
\end{proof}
\subsection{Higher dimensional situation}
The following example and Theorem \ref{thm-balls}, show that the situation in higher 
dimensions is different.

\begin{example} \label{ex-hartogsfigure}
Let $\Omega \subset \cx^2$ be the bidisc, $0<\rho_1,\rho_2<1$, and  
\[ U =U_1\times U_2=\{z\in \Omega: \abs{z_1}<\rho_1, \rho_2<\abs{z_2}<1\}\]
be the product of  the disc  $U_1= \{z\in \cx: \abs{z}<\rho_1\}$  and 
the annulus $U_2= \{z\in \cx : \rho_2<\abs{z}<1\}$. 
The eigenvalues of the associated Toeplitz operator $T_U$ of the restriction operator
$R_U:A^2(\Omega)\to A^2(U)$ can be found using formula \eqref{eq-eigenvaluealpha}, 
which are given by ($\alpha\in \mathbb{Z}_+^2$)
\[ \lambda_\alpha
= \frac{\norm{z^\alpha}_U^2}{\norm{z^\alpha}_\Omega^2}
=\rho_1^{2\alpha_1+2}(1-\rho_2^{2\alpha_2+2}).\]
Consequently, we have by \eqref{eq-reinhardtnorm}
\[\norm{R_U}
= \sup_{\alpha\in \mathbb{Z}_+^2}\sqrt{ \rho_1^{2\alpha_1+2}(1-\rho_2^{2\alpha_2+2})}
= \rho_1.\]
This example shows that Proposition \ref{prop-generalplanar} does not 
hold in  higher dimensions, and in fact it is possible for  $\Omega$ and 
$U$ to share an open subset of the boundary and still have $\|R_U\|<1$.

Note that $V=\Omega \setminus \ol{U}$ is the well-known Hartogs figure,  
and it is not difficult to see that the map $R_V:A^2(\Omega)\to A^2(V)$ 
is surjective.
\end{example}

We close this section with the proof of Theorem \ref{thm-balls}. 
\begin{proof}[Proof of Theorem \ref{thm-balls}]
Without any loss of generality, we can assume that $R=1$. Taking 
$f\equiv 1$, we see that
\[ \norm{f}_\Omega^2= \int_\Omega 1 dV= {\rm Vol}(\Omega),
\text{ and }  \norm{f}_U^2 = \int_U 1 dV ={\rm Vol}(U).\]
Then we have 
$\frac{\norm{f}_U^2}{\norm{f}_\Omega^2}=r^{2n} \leq \norm{R_U}^2$, 
which establishes the lower bound in \eqref{eq-ruestimate}. The fact 
that this norm is attained when $U$ and $\Omega$ are concentric 
follows from \eqref{eq-normru}.

When $n=1$, the right hand side of \eqref{eq-ruestimate} is 1, 
so there is nothing to prove. Consequently, we will assume $n\geq 2$.  
We denote the coordinates of $\cx^n$ by $(z,w)$, where $z\in \cx$ and 
$w\in \cx^{n-1}$. After a coordinate change given by unitary rotation 
and translation, we may suppose that 
\[ \Omega=\{(z,w)\in \cx\times \cx^{n-1}: \abs{z-1}^2 + \abs{w}^2 <1\},\]
and
\[ U = \{(z,w)\in \cx\times \cx^{n-1}: \abs{z-\delta}^2 + \abs{w}^2 <r^2\}.\]
Clearly 
\begin{equation}\label{eq-c}
 r\leq \delta\leq 1.
\end{equation}
  Denoting by $\pi$ the projection 
from $\cx^n$ to $\cx$ given by $(z,w)\to z$, we see that 
$\pi(U)=\{z\in\cx: \abs{z-\delta}<r\}$, which is a smaller 
disc contained in the  disc $\pi(\Omega)=\{z\in \cx:\abs{z-1}<1\}$. 
For $z\in \pi(U)$, we note that $\pi^{-1}(z)\cap U =\{z\}\times \Delta_U(z)$, 
where $\Delta_U(z)$ is the ball in $\cx^{n-1}$ defined by
\[ \Delta_U(z)
=B_{\cx^{n-1}}\left(0, \sqrt{r^2 - \abs{z-\delta}^2}\right)
=\left\{w\in\cx^{n-1}: \abs{w}^2<r^2 - \abs{z-\delta}^2\right\}.\]
Similarly, for a $z$ in the unit disc, let 
$\pi^{-1}(z)\cap \Omega=\{z\}\times \Delta_\Omega(z)$, 
where $\Delta_\Omega(z)$ is the ball in $\cx^{n-1}$ given by
\[ \Delta_\Omega(z)
=B_{\cx^{n-1}}\left(0, \sqrt{1 - \abs{z-1}^2}\right)
=\left\{w\in\cx^{n-1}: \abs{w}^2<1 - \abs{z-1}^2\right\}.\]
With this notation, we have for any integrable function $u$ the formulas
\[\int_\Omega u dV 
= \int_{\pi(\Omega)} \left(\int_{\Delta_\Omega(z)} u(z,w) dV(w)\right)dV(z) 
\text{ and }   \int_U u dV 
= \int_{\pi(U)} \left(\int_{\Delta_U(z)} u(z,w)dV(w)\right)dV(z),\] 
by representing the integral as a repeated integral.

For notational clarity,  for $z$ in the disc $\pi(U)$, let $S(z)=R_{\Delta_U(z)}$ denote 
the restriction operator 
\[ S(z)= R_{\Delta_U(z)}:A^2(\Delta_\Omega(z))\to A^2(\Delta_U(z)).\]
Since $\Delta_\Omega(z)$ and $\Delta_U(z)$ are concentric balls in $\cx^{n-1}$, 
it follows from \eqref{eq-normru} that 
\begin{equation}\label{eq-normsz}
 \norm{S(z)}= \left( \frac{r^2-\abs{z-\delta}^2}{1-\abs{z-1}^2}\right)^{\frac{n-1}{2}}.
 \end{equation}

 We have, for each $f\in A^2(\Omega)$
 \begin{align*}
 \norm{f}_U^2 &=\int_{\pi(U)} \left(\int_{\Delta_U(z)} \abs{f(z,w)}^2dV(w)\right)dV(z)\\
 &= \int_{\pi(U)}\norm{f(z,\cdot)}^2_{\Delta_U(z)} dV(z)\\
 &\leq  \int_{\pi(U)} \norm{S(z)}^2 \norm{f(z,\cdot)}^2_{\Delta_\Omega(z)} dV(z)\\
 &\leq \sup_{z\in \pi(U)} \norm{S(z)}^2  
 \int_{\pi(U)}\left(\int_{\Delta_\Omega(z)} \abs{f(z,w)}^2 dV(w)\right)dV(z)\\
  &<\sup_{z\in \pi(U)} \norm{S(z)}^2  
  \int_{\pi(\Omega)}\left(\int_{\Delta_\Omega(z)} \abs{f(z,w)}^2 dV(w)\right)dV(z)\\
  &=\left(\sup_{\abs{z-\delta}<r} \norm{S(z)}^2\right) \norm{f}_\Omega^2.
 \end{align*}
 So by \eqref{eq-normsz} we have
 \[ \norm{R_U} \leq \sup_{\abs{z-\delta}<r} \norm{S(z)} 
 =\sup_{\abs{z-\delta}<r}\left( \frac{r^2-\abs{z-\delta}^2}{1-\abs{z-1}^2}\right)^{\frac{n-1}{2}}.\]
Notice now that thanks to \eqref{eq-c}, we have
\[ (\delta^2-r^2)+(1-\delta)\abs{z}^2 \geq 0.\]
Since we can write
\[ r^2-\abs{z-\delta}^2 
= \delta(1-\abs{z-1}^2) - \left( (\delta^2-r^2)+(1-\delta)\abs{z}^2\right)\]
it follows that
\[ \frac{r^2-\abs{z-\delta}^2}{1-\abs{z-1}^2} \leq \delta\]
for each $z$ such that $1-\abs{z-1}^2>0$ (and therefore for $z$ in the smaller 
set $r^2-\abs{z-\delta}^2>0$).  It now follows that
\[ \norm{R_U} \leq \delta^{\frac{n-1}{2}}, \]
thus establishing the upper bound in \eqref{eq-ruestimate}.

To complete the proof, we need to show that if  $\delta=r$, then 
we have $\norm{R_U}\geq r^{\frac{n-1}{2}}$, for $n\geq 1$.  Note that 
since the disc $\pi(\Omega)= \{z\in \cx:\abs{z-1}<1\}$ is contained in the right half plane, 
for each $\gamma>0$, we can define a branch of $z^{-\gamma}$ on $\Omega$. 
Let $f_\gamma$ be such a branch, normalized for uniqueness by the condition 
$f_\gamma(1)=1$. Observe that 
\begin{align*}
\norm{f_\gamma}^2_U &= \int_{\pi(U)}\int_{\Delta_U(z)}\abs{f_\gamma(z,w)}^2dV(w)dV(z)\\
&=\int_{\{\abs{z-r}<r\}} \abs{z}^{-2\gamma} {\rm Vol}
\left(B_{\cx^{n-1}}\left(0, \sqrt{r^2 - \abs{z-r}^2}\right)\right) dV(z)\\
&= {\rm Vol}\left(B_{\cx^{n-1}}\left(0,1\right)\right)\int_{\{\abs{z-r}<r\}} \abs{z}^{-2\gamma} 
\left( r^2 - \abs{z-r}^2\right)^{n-1}dV(z),
\end{align*}
where volumes of zero-dimensional balls are defined to be 1. 
We now make a change of variables in the above integral to a new variable $w$, 
related to $z$ by
\[ z =r w,\]
so that by the change of variables formula, the above integral becomes
\begin{align}
\norm{f_\gamma}^2_U &= {\rm Vol}\left(B_{\cx^{n-1}}\left(0,1\right)\right)
\int_{\{\abs{w-1}<1\}} r^{-2\gamma}\abs{w}^{-2\gamma} r^{2(n-1)}
\left( 1 - \abs{w-1}^2\right)^{n-1}r^2dV(w)\nonumber\\
&=r^{2n-2\gamma} {\rm Vol}
\left(B_{\cx^{n-1}}\left(0,1\right)\right)\left(\int_{\{\abs{w-1}<1\}} \abs{w}^{-2\gamma}
\left( 1 - \abs{w-1}^2\right)^{n-1}dV(w)\right)\label{eq-fgammanorm}\\
&= r^{2(n-\gamma)} \norm{f_\gamma}_\Omega^2.\nonumber
\end{align}
Therefore, for each $\gamma$ such that $f_\gamma\in A^2(\Omega)$ we have
$\norm{R_U}\geq r^{n-\gamma}.$ We now claim that if $\gamma< \frac{n+1}{2}$, then 
$f_\gamma\in A^2(\Omega)$.  Assuming  the claim, we conclude that 
\[ \norm{R_U}\geq \sup_{\gamma<\frac{n+1}{2}}r^{n-\gamma}=r^{\frac{n-1}{2}},\]
which completes the proof of the result modulo the claim. To prove the 
claim, it suffices to show that the integral in \eqref{eq-fgammanorm} is finite 
if $\gamma< \frac{n+1}{2}$. In fact, since
\[ 1-\abs{w-1}^2 =2\Re w -\abs{w}^2 \leq 2 \Re w,\]
it follows that it suffices to show that the integral 
\[ \int_{\{\abs{w-1}<1\}} \abs{w}^{-2\gamma}
\left( \Re w\right)^{n-1}dV(w)\]
is finite. Switching to polar coordinates $w=\rho e^{i\theta}$, this integral becomes
\begin{align*}
\int_{\theta=-\frac{\pi}{2}}^{\frac{\pi}{2}} \int_{\rho=0}^{2\cos \theta} \rho^{-2\gamma}
\left(\rho\cos\theta\right)^{n-1}\rho d\rho d\theta
&= \int_{-\frac{\pi}{2}}^{\frac{\pi}{2}} \left(\cos\theta\right)^{n-1} 
\left(\int_{0}^{2\cos \theta}\rho^{n-2\gamma} d\rho\right) d\theta\\
&= \int_{-\frac{\pi}{2}}^{\frac{\pi}{2}}\left(\cos\theta\right)^{n-1}
 \left.\frac{\rho^{n-2\gamma+1}}{n-2\gamma+1}\right\vert_0^{2\cos \theta} d\theta.
\end{align*}
We note that $\gamma< \frac{n+1}{2}$ implies that the integral 
$\int_{0}^{2\cos \theta}\rho^{n-2\gamma} d\rho$ converges at the endpoint 0 as 
$n-2\gamma>-1$. Therefore, up to some irrelevant multiplying constants, the last
integral above becomes
\[ \int_{-\frac{\pi}{2}}^{\frac{\pi}{2}}\left(\cos\theta\right)^{n-1}
\left(\cos\theta\right)^{n-2\gamma+1} d\theta 
=\int_{-\frac{\pi}{2}}^{\frac{\pi}{2}} \left(\cos\theta\right)^{2(n-\gamma)} d\theta.\]
Since $\gamma<\frac{n+1}{2}$, it follows that $2(n-\gamma)>n-1\geq 0$, 
which shows that the integrand in the above integral is  continuous. 
This shows that $f_\gamma\in A^2(\Omega)$.
\end{proof}

\section{Proof of Theorems \ref{thm-horostrip} and  \ref{thm-hypercycle}}
\label{sec-horosymmetry}
\subsection{Two results of N. Vasilevski}
We will deduce  Theorems \ref{thm-horostrip} and \ref{thm-hypercycle} 
as special cases of the following determination of spectra of Toeplitz operators
on the Bergman space of the upper half plane $\mathbb{H}=\{z\in \cx| \Im(z)>0\}$. 
Recall that given a function $\psi\in L^\infty(\mathbb{H})$, the 
\textit{Toeplitz operator} $T_\psi$ with \textit{symbol} $\psi$ is the bounded 
linear operator on $A^2(\mathbb{H})$ defined as
\[ T_\psi f = P(\psi f)\]
where $P:L^2(\mathbb{H}) \to A^2(\mathbb{H})$ is the Bergman projection on 
$A^2(\mathbb{H})$. Recall also that given a measure space $(X, \mathcal{F}, \mu)$ 
and a function $\gamma\in L^\infty(\mu)$, the multiplication operator $M_\gamma$ 
on $L^2(\mu)$ with \textit{multiplier} $\gamma$ is the bounded operator defined 
by multiplication by $\gamma$:
\[ M_\gamma f =\gamma f.\]
\begin{theorem}[{\cite[Theorem 5.2.1]{vasilevski}}]\label{thm-vasilevski1} 
	Let $\rl^+=(0,\infty), \psi\in L^{\infty}(\mathbb{H}),$ and $a\in L^{\infty}(\rl^+)$ 
	such that $\psi(z)=a(\Im z)$ for all $z\in \mathbb{H}$. Then the Toeplitz 
	operator $T_\psi$ on $A^2(\mathbb{H})$ is unitarily equivalent to the 
	multiplication operator on $L^2(\rl^+)$ with multiplier $\gamma$ given by
\begin{equation}\label{eq-gamma1}
\gamma(x) = \int_{\rl^+} a\left(\frac{\eta}{2x}\right)e^{-\eta}d\eta.
\end{equation}
\end{theorem}
\begin{theorem}[{\cite[Theorem 7.2.1]{vasilevski}}]\label{thm-vasilevski2} 
	Let  $\psi\in L^{\infty}(\mathbb{H})$ and $a\in L^{\infty}(0, \pi)$ such that 
	$\psi(re^{i\theta})=a(\theta)$ for all $0<\theta<\pi$. Then the Toeplitz 
	operator $T_\psi$ on $A^2(\mathbb{H})$ is unitarily equivalent to the 
	multiplication operator on $L^2(\rl)$ with multiplier $\gamma$ given by
\begin{equation}\label{eq-gamma2}
\gamma(\lambda)
= \frac{2\lambda}{1-e^{-2\pi\lambda}} \int_0^{\pi} a(\theta) e^{-2\lambda\theta} d\theta.
\end{equation}
\end{theorem}
Such remarkable explicit determination of the spectrum of a Toeplitz 
operator is terms of the symbol is quite rare and can be achieved only on 
very symmetric special cases as above. 

\subsection{Mapping to the upper half-plane}\label{sec-mapping} 
In order to apply Theorems \ref{thm-vasilevski1} and \ref{thm-vasilevski2} to the 
proofs of Theorems \ref{thm-horostrip} and \ref{thm-hypercycle}, respectively, 
we begin by mapping the unit disc $\disk$  conformally to the upper half 
plane  $\mathbb{H}$. Thanks to Proposition \ref{prop-invariance} we know 
that the spectrum of $T_U$ is invariant under biholomorphisms of the ambient 
domain, so we can use this upper half-plane model to compute the spectrum.

Consider, as in Theorem \ref{thm-horostrip}, part 2, a horocyclic strip in the 
unit disc, where the outer and inner bounding horocycles have Euclidean radii 
$\rho_2$ and $\rho_1$ respectively. After a rotation of the unit disc, we may 
assume that the point of contact is at $1\in \partial \mathbb{D}$.  By the 
standard conformal map of $\mathbb{D}$ to $\mathbb{H}$ given by 
\begin{equation}\label{eq-cayley}
 z\mapsto  i \frac{1+z}{1-z}
\end{equation}
the bounding horocycle $\{z\in \cx: \abs{z-(1-\rho_j)} = \rho_j\}$ is mapped to 
a Euclidean straight line in the upper half plane parallel to the real axis, 
given by $\{ z\in \cx: \Im z = \frac{1}{\rho_j} -1\}$. Therefore, the horocyclic strip 
bounded by  Euclidean circles of radii $0<\rho_1<\rho_2<1$ is therefore 
mapped into the horizontal strip in the upper half plane given by
\[ \left\{ z\in \cx: \frac{1}{\rho_2}-1 < \Im z < \frac{1}{\rho_1}-1\right\}.\]
Applying the dilation of the upper half plane given by 
$z\mapsto  \frac{z}{\frac{1}{\rho_2}-1}$, this is mapped onto 
the horizontal strip
\[ V_\alpha = \{z\in \cx:  1<\Im  z <\alpha\},\]
where $\alpha$ is as in \eqref{eq-alpha}.  The full horodisc corresponds 
to the limiting situation of $\rho_1\to 0$, and therefore can be mapped 
to the half-plane
\[ V_\infty = \{ z\in \cx: \Im z >1\}.\]

We now consider the mapping of hypercycles and lunes and crescents 
determined by them. Let $H$ be a hypercycle in the Poincaré disc with axis 
$\Gamma$.  Recall that by definition this means that $H$ is a connected curve 
consisting of points at a fixed hyperbolic distance from $\Gamma$. After  applying  
an automorphism of  the unit disc, the geodesic $\Gamma$ is mapped onto the 
interval $(-1,1)$ on the real axis. The hypercycle $H$ is mapped to an arc of a 
circle passing through the points $\pm 1$ (see Figure \ref{fig-map-hypercycle}). 
\begin{figure}
		\begin{center}
			\begin{tikzpicture}[scale=2.5]
			\draw [dashed] (0,0) circle (1cm); 
			\draw (1,0) arc [radius=1.11803, 
			start angle=26.56505, end angle= 153.43494];
			\draw [dashed] (-1,0)--(1,0);
			\node [below] at (0,0) {$\Gamma$};
			\node [above] at (0, 0.60) {$H$};
			\end{tikzpicture}
			\hspace{1.5cm}
			\begin{tikzpicture}[scale=3]
			\draw [dashed] (-1,0)--(1,0);
			\draw [dashed] (0,0)-- (0,1.6);
			\draw (0,0)--(1,0.6954);
			\node [left] at (0,1) {$\Gamma$};
			\node [below] at (0.5, 0.30) {$H$};
			\end{tikzpicture}
		\end{center}
\caption{Hypercycles in  the disc and the upper half-plane}\label{fig-map-hypercycle}
\end{figure}
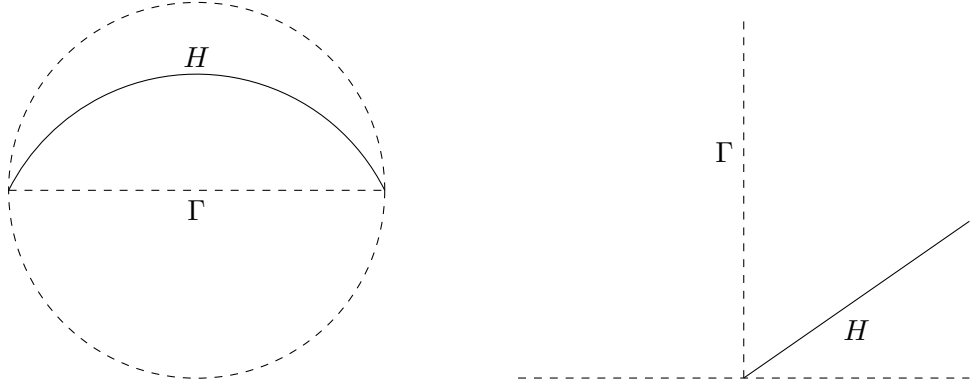
Under the conformal map \eqref{eq-cayley}, $-1$ is mapped to 0, $1$ is mapped to 
$\infty$ and therefore the diameter $\Gamma=(-1,1)$ is mapped to the positive 
imaginary axis. It follows that the hypercycle $H$ (which is an arc of an Euclidean 
circle in the complex plane) is mapped to a rectilinear ray in the upper half plane 
$\mathbb{H}$ passing through 0. Therefore, a hypercyclic lune is represented in
 the upper half plane by a wedge of the form
\begin{equation}\label{eq-wedge}
W(\alpha, \beta) 
= \{re^{i\theta}\in \mathbb{H}|r>0 \text{ and } \alpha<\theta<\beta\},
\end{equation}
where $0<\alpha<\beta<\pi$. Similarly, a hypercyclic crescent takes the 
form $W(\alpha,\pi)$ or $W(0, \beta)$.

\begin{proof}[Proof of Theorem \ref{thm-horostrip}]
Thanks to Proposition \ref{prop-invariance} and the conformal mappings 
constructed in Section \ref{sec-mapping},  the spectrum of a horodisc in 
$\mathbb{D}$ coincides with that of the operator $T_{V_\infty}$ on 
$A^2(\mathbb{H})$ where $V_\infty = \{ z\in \cx: \Im z >1\}.$  It follows easily 
from Theorem \ref{thm-vasilevski1}, the spectrum of 
$T_{V_\infty}$ is the interval
\begin{equation}\label{eq-interval}
\left[ \inf_{x\in \rl^+}\gamma(x), \sup_{x\in \rl^+}\gamma(x)\right]
\end{equation}
where $\gamma$ is as in \eqref{eq-gamma1}, with  $a = \chi_{[1,\infty)}$,
the characteristic function of the interval $[1,\infty)$.  Therefore, for $x>0$, we have
\begin{align*}
	\gamma(x)=\int_{\frac{\eta}{2x}>1}e^{-\eta}d\eta
= \int_{2x}^\infty e^{-\eta}d\eta
=e^{-2x}.
\end{align*}
Therefore $\inf_{x\in \rl^+}\gamma(x)=0$ and $\sup_{x\in \rl^+}\gamma(x)=1$, and part (1) 
of Theorem \ref{thm-horostrip} follows.

For part (2),  note that the spectrum is still given by \eqref{eq-interval}, provided 
we define $\gamma$ in \eqref{eq-gamma1} with $a=\chi_{[1,\alpha]}$. Therefore, 
in this case, for each $x>0$.
\begin{align*}
	\gamma(x)=\int_{1<\frac{\eta}{2x}<\alpha}e^{-\eta}d\eta 
=\int_{2x}^{2x\alpha}e^{-\eta}d\eta
= e^{-2x}-e^{-2x\alpha}.
\end{align*}
But as $\lim_{x\to \infty}  \left(e^{-2x}-e^{-2x\alpha}\right)=0$, it follows that 
$\inf_{x\in \rl^+}\gamma(x)=0$. To find the supremum, note that 
$\gamma'(x)= -2e^{-2x} + 2 \alpha e^{-2x\alpha}$, so that the only 
critical point of $\gamma$  is given by
\[ x_{\rm crit} = \frac{1}{2} \frac{\ln \alpha}{\alpha-1}.\]
Note also that $\lim_{x\to 0} \gamma(x)=0$, and we must have 
$0\leq \gamma \leq 1$ (since the range of $\gamma$ coincides 
with the spectrum of $T_{V_\alpha}$, which we know to be
a subset of $[0,1]$.) So it follows that the critical point $x_{\rm crit}$ 
is in fact a point of global maximum of $\gamma$ on $\rl^+$. Therefore
\begin{align*}
 \sup_{x\in \rl^+}\gamma(x)= \gamma(x_{\rm crit})
 = \alpha^{- \frac{1}{\alpha -1}} - \alpha^{- \frac{\alpha}{\alpha-1}},
 \end{align*}
 which completes the proof of Theorem \ref{thm-horostrip}.
\end{proof}

\begin{proof}[Proof of Theorem \ref{thm-hypercycle}]
By biholomorphic invariance, it suffices to consider the operator 
$T_{W(\alpha,\beta)}$ on $A^2(\mathbb{H})$. Thanks to 
Theorem \ref{thm-vasilevski2}, the spectrum of 
$T_{W(\alpha,\beta)}$ coincides with the closure of range 
of the function $\gamma:\rl \to \rl$ given by
 \begin{align}\label{eq-xi}
\gamma(\lambda)= \frac{2\lambda}{1-e^{-2\pi\lambda}} 
 \int_{\alpha}^\beta e^{-2\lambda\theta}d\theta
 = \frac{e^{-2\lambda\alpha}-e^{-2\lambda\beta}}{1-e^{-2\pi\lambda}}
 =\frac{\xi^b-\xi^a}{\xi-1}
 \end{align}
 where $\xi=\xi(\lambda)=e^{-2\pi\lambda}, a = \frac{\alpha}{\pi}, b=\frac{\beta}{\pi}$, 
 so that $0\leq a < b\leq {1}$. Note that  $\xi\to 0$ as $\lambda\to \infty$ and 
  $\xi\to \infty$ as  $\lambda \to -\infty$. 
 
 First consider the situation of a hypercyclic crescent represented in 
 $\mathbb{H}$ by $W(\alpha, \pi)$, where $0<\alpha<\pi$. Then, in 
 \eqref{eq-xi} we have $b=1,0<a<1$, and $\gamma= \frac{\xi-\xi^a}{\xi-1}$. 
 So $\lambda\to\infty$ implies that $\xi\to0$ which, in turn implies that 
 $\gamma\to 0$. On the other hand, $\lambda\to -\infty$ implies that 
 $\xi \to \infty$ and hence $\gamma\to 1$. Therefore we have
 \[ {\rm spec}(T_{W(\alpha,\pi)}) 
 =\left[ \sup_{\lambda\in \rl} \gamma(\lambda), 
 \inf_{\lambda\in \rl} \gamma(\lambda)\right]
 =[0,1].\]
 The other possibility for a hypercyclic crescent is to be represented 
 in $\mathbb{H}$ by the wedge $W(0,\beta)$, where $0<\beta<1$. 
 Then in \eqref{eq-xi}, we have $a=0$ and $0<b<1$. So we have 
 $\gamma = \frac{\xi^b-1}{\xi-1}$. Therefore as $\lambda\to \infty$ 
 we have $\xi\to 0$ and $\gamma\to 1$. Similarly as $\lambda\to -\infty$ 
 we have $\xi\to \infty$ and $\gamma\to 0$. This shows again that the 
 spectrum of $T_{W(0,\beta)}$ equals to $[0,1]$. This concludes the 
 proof of Part (1) of Theorem \ref{thm-hypercycle}.

For part (2) of Theorem \ref{thm-hypercycle},  if $0<a<b<1$, then the function
\[ \frac{\xi^b-\xi^a}{\xi-1}\]
goes to 0 as either $\xi\to \infty$ or $\xi\to 0$ (i.e., $\lambda\to \pm \infty$). 
It follows that the spectrum of $T_{W(\alpha,\beta)}$ is of the form 
$[0, c(\alpha,\beta)]$, where
\[ c(\pi a, \pi b)=c(\alpha,\beta)
= \norm{T_{W(\alpha,\beta)}} 
= \sup_{0<\xi<\infty} \frac{\xi^b-\xi^a}{\xi-1}.\]
The function $c(\alpha,\beta)$ cannot be expressed in closed form. 
But note that for fixed $0<\alpha<\pi$ as $\beta\to \alpha^+$ clearly 
$c(\alpha,\beta)\to 0$ and as $\beta\to \pi$, $c(\alpha,\beta)\to 1$, 
since as $b\to a$, the function $\xi\mapsto \frac{\xi^b-\xi^a}{\xi-1}$
converges uniformly to 0 on compact sets, and as $b\to 1$, it converges 
uniformly to 1 on compact sets. To complete the proof we need to  
show that $c(\alpha,\beta)=\norm{T_{W(\alpha,\beta)}}<1$ 
This follows from an elementary computation which is given 
below as a separate Lemma. 
 \end{proof}

 \begin{lemma} Let $0<a<b<1$. Then 
 \[c(\pi a, \pi b)= \sup_{0<\xi<\infty} \frac{\xi^b-\xi^a}{\xi-1}<1.\]
 \end{lemma}
 \begin{proof} 
 	Let $f(\xi)= \dfrac{\xi^b-\xi^a}{\xi-1}$. As $\xi\to 0$, we have $f(\xi)\to 0$. 
 	Therefore, there is a $\delta>0$ such that
 \begin{equation}\label{eq-sup1}
\sup_{0<\xi\leq\delta} \frac{\xi^b-\xi^a}{\xi-1}< \frac{1}{2}.
\end{equation}
  Since $0<a<b<1$, we have, for $\delta< \xi < 1$ that
 \[ \delta< \xi < \xi^b <\xi^a < 1.\]
and therefore that
 \begin{align*}
 f(\xi) = \frac{1-\xi^b}{1-\xi} - \frac{1-\xi^a}{1-\xi} < \frac{1-\xi^b}{1-\xi} <1.
\end{align*}
 Combined with the facts that $\lim_{\xi\to \delta}f(\xi)=f(\delta)<\frac{1}{2}$ and 
 \[ \lim_{\xi\to 1}f(\xi) = \lim_{\xi\to 1}\frac{b\xi^{b-1}-a\xi^{a-1}}{1}=b-a<b<1,\]
 we obtain that
 \begin{equation}\label{eq-sup2}
\sup_{\delta\leq\xi\leq1} \frac{\xi^b-\xi^a}{\xi-1}< 1.
\end{equation}
Furthermore, if $\xi>1$ then $1<\xi^a< \xi^b<\xi$. So we have
 \begin{align*}
f(\xi) = \frac{\xi^b- 1}{\xi-1}- \frac{\xi^a-1}{\xi-1}
 < \frac{\xi^b- 1}{\xi-1}<1.
 \end{align*}
 Therefore, $f(\xi)<1$ for each $\xi\in [1,\infty)$. Further, since $\lim_{\xi\to \infty} f(\xi)=0$, 
 it follows that 
 \begin{equation}\label{eq-sup3}
\sup_{1\leq \xi<\infty} \frac{\xi^b-\xi^a}{\xi-1} <1.
\end{equation}
Combining \eqref{eq-sup1}, \eqref{eq-sup2}, \eqref{eq-sup3} the result follows. 
  \end{proof}

\subsection{Some examples and comments}
\subsubsection{Isospectrality of hypercyclic crescents and horodiscs} 
Theorems \ref{thm-horostrip} and \ref{thm-hypercycle} show that a horodisc 
and a hypercyclic crescent are isospectral, both having spectrum $[0,1]$. 
But they are not congruent to each other (i.e., there is no automorphism 
of $\disk$ which maps a hypercyclic crescent to a horodisc). Indeed, the 
boundary relative to the unit disc of a hypercyclic crescent is a hypercycle, 
whereas that of a horodisc is a horocycle. But  the geodesic curvature of a 
hypercycle is less than $2\sqrt{\pi}$ in absolute value, whereas that of a 
horocycle is equal to $2\sqrt{\pi}$ in absolute value 
(see \cite[Section 11.6, pg 282]{vasilevski}). 
Since the automorphisms are isometries of the Poincar\'e metric, 
all quantities defined in terms of the metric (such as curvature) 
are  invariant. Therefore, horocycles and hypercycles are not congruent.

\subsubsection{Isospectral hypercyclic lunes and horocyclic strips}
Another example of such isospectral but noncongruent subdomains of the 
disc is obtained by looking at two hypercyclic crescents  $U_1$ and $U_2$ 
with coaxial but distinct bounding hypercycles $H_1$ and $H_2$, since 
$H_1$ and $H_2$ have different geodesic curvatures with respect to 
the hyperbolic metric (see \cite[Theorem 11.6.2]{vasilevski}).
Note also that given any $0<c<1$, there is  a horocyclic strip and a 
hypercyclic lune (which are non-congruent thanks to the curvature 
argument above), such that each has spectrum $[0,c]$. This follows  
from a closer analysis of the formulas in part (ii) of each of 
Theorems \ref{thm-horostrip} and \ref{thm-hypercycle}.

\subsubsection{Examples of closed range} 
Consider a horocyclic strip $H$ in the disc $\disk$, like the shaded region in 
Figure \ref{fig-horostrip}. Then the complement $U$ of $H$ can be written 
as a disjoint union
\[ U=\disk \setminus H = U_1\cup U_2,\]
where $U_1$ is a horodisc and $U_2$ is the complement of a horodisc 
(therefore $U_2$ is bounded by a horocycle and the unit circle.) From 
Proposition \ref{prop-eigen} we see that the spectrum of $U_2$ relative 
to $\disk$ is the interval $[0,1]$. Since the spectrum of $H$ is of the form
$[0,\norm{T_H}]$, where $0<\norm{T_H}<1$, we see again from 
Proposition \ref{prop-eigen} that 
\[ {\rm spec}(T_U)= [1-\norm{T_H},1]\]
so that in particular, by Corollary \ref{prop-closedrange}, the restriction 
operator on  $U=U_1\cup U_2$ has closed range. 

Another similar example is the complement of the hypercyclic lune $W$ in 
Figure \ref{fig-hypercycle}. Since again the spectrum of $W$ relative to 
$\disk$ is the interval $[0, \norm{T_W}]$  it follows that the complement 
$\disk\setminus \overline{W}$ has spectrum
$[1-\norm{T_W},1]$, and the restriction operator from $A^2(\disk)$ 
to $A^2(\disk\setminus\overline{W})$ has closed range.

\section{Proof of Theorem \ref{thm-symmetry}}\label{section-symmetry}
We will use the following fact  (see \cite[Theorem 1.3.12]{green-krantz}) 
in the proof of Theorem \ref{thm-symmetry}:  \textit{If $\Omega$ is a 
	bounded domain in $\cx^n$, the map 
\begin{align*}
{\rm Aut}(\Omega)\times \Omega &\to \Omega\times \Omega\\
(\phi,z)&\mapsto (\phi(z),z)
\end{align*}
is proper.} Recall that to say that a continuous map between topological 
spaces is \textit{proper} means that the inverse image of each compact 
subset of the codomain is compact in the domain of the map. Recall also 
that the topology on $\textrm{Aut}(\Omega)$ is the natural compact-open 
topology. 

\begin{proof}[Proof of Theorem \ref{thm-symmetry}]
Since $G$ is a closed subset of ${\rm Aut}(\Omega)$, it follows from  
\cite[Theorem 1.3.12]{green-krantz} that the restricted map
\begin{align} 
G\times \Omega &\to \Omega\times \Omega \label{eq-action}\\
 (\phi,z)&\mapsto (\phi(z),z)\nonumber
 \end{align}
is also proper. Let $z_0\in \Omega$ and consider the orbit of $p$ under 
$\Omega$ i.e. the set
\[ G(z_0)=\{\phi(z_0)|\phi\in G\}.\]
We claim that $G(z_0)$ is noncompact. Indeed, if it were compact, its 
inverse image under the map \eqref{eq-action}  would be compact, i.e., 
$G\times\{z_0\}$ would be compact; but by hypothesis $G$ is noncompact, 
so we have a contradiction. Choosing $z_0\in U$,  and noting that by 
hypothesis $G(z_0)\subset U$, we see that there is a point 
$q\in \partial \Omega\cap \partial U$, and a sequence of automorphisms 
$\{\Phi_j\}\subset G$ such that $\Phi_j(z_0)\to q$ as $j\to \infty$. 
Consider the functions $f_j\in A^2(\Omega)$ defined by
\[ f_j(z)= \frac{B_\Omega(z,\Phi_j(z_0))}{\sqrt{B_\Omega(\Phi_j(z_0),\Phi_j(z_0))}}.\]
We begin by noting (using the reproducing property of the Bergman kernel) that
\begin{align*}
 \norm{f_j}_{A^2(\Omega)}
 &=\frac{1}{B_\Omega(\Phi_j(z_0),\Phi_j(z_0))}
 \int_\Omega B_\Omega(z,\Phi_j(z_0))B_\Omega (\Phi_j(z_0),z)dV(z)\\
 &= \frac{1}{B_\Omega(\Phi_j(z_0),\Phi_j(z_0))}\cdot B_\Omega(\Phi_j(z_0),\Phi_j(z_0))\\
 &=1.
 \end{align*}
 Assume now for a contradiction that $T_U$ is compact, so that the restriction 
 operator $R_U:A^2(\Omega)\to A^2(U)$ is also compact.  Consequently, after 
 passing to a subsequence  if necessary, we may assume that the sequence 
 $R_U f_j=f_j|_U$ converges in the $L^2$-norm as $j\to \infty$ to a $g\in A^2(U)$. 
 Now, note that 
 \begin{align*}
 \int_U \abs{f_j}^2 dV 
 &= \frac{1}{B_\Omega(\Phi_j(z_0),\Phi_j(z_0))}\int_U \abs{B_\Omega(z,\Phi_j(z_0))}^2 dV(z)\\
 &=  \frac{1}{B_\Omega(\Phi_j(z_0),\Phi_j(z_0))\cdot\abs{\det \Phi_j'(z_0)}^2}
 \int_U \abs{B_\Omega(z,\Phi_j(z_0))}^2\cdot\abs{\det \Phi_j'(z_0)}^2 dV(z)\\
 &= \frac{1}{B_\Omega(z_0,z_0)} \int_U \abs{B_\Omega(\Phi_j(w),\Phi_j(z_0))}^2
 \cdot\abs{\det \Phi_j'(w)}^2\cdot\abs{\det \Phi_j'(z_0)}^2 dV(w)\\
 &=\frac{1}{B_\Omega(z_0,z_0)} \int_U \abs{\ol{\det \Phi_j'(w)}
 	\cdot B_\Omega(\Phi_j(w),\Phi_j(z_0))\cdot\det \Phi_j'(z_0)}^2 dV(w)\\
 &=\frac{1}{B_\Omega(z_0,z_0)}\int_U \abs{B_\Omega(w,z_0)}^2 dV(w)\\
 &= C^2,
 \end{align*}
where $C>0$ is independent of $j$, and during the course of the computation we 
have used twice the transformation formula for the Bergman kernel under a 
biholomorphic map. Therefore, the $L^2$-limit $g$ of the sequence $f_j$ also satisfies
$\norm{g}_{A^2(U)}= C$. On the other hand, we claim that $f_j\to 0$ weakly as 
$j\to \infty$, which contradicts the  fact that $f_j$ tends to a nonzero limit in the $
L^2$-norm. This contradiction shows that $R_U$ and therefore $T_U$ is not compact. 

 To complete the proof we justify the claim that $f_j\to 0$ weakly as $j\to \infty$.  
 Since $\D$ is smoothly bounded and pseudoconvex by a result of Catlin 
 \cite[Theorem 3.2.1]{Catlin80}, $A^{\infty}(\Dc)$, the space of functions 
 holomorphic on $\D$ and smooth up to the boundary, is dense in $A^2(\D)$.  
 Fix $h\in A^2(\Omega)$, and let $\ep>0$ be given. Then there exists 
 $h_{\delta}\in A^{\infty}(\Dc)$ such that $\|h-h_{\delta}\|<\ep$. Then  
 \[|\langle h,f_j\rangle| \leq |\langle h-h_{\delta},f_j\rangle|+|\langle 
	h_{\delta},f_j\rangle|\leq \|h-h_{\delta}\|+|\langle h_{\delta},f_j\rangle|
	<\ep+|\langle h_{\delta},f_j\rangle|\]
However, we note that 
$\langle h_{\delta},f_j\rangle
=h_{\delta}(\Phi_j(z_0))/\sqrt{B_\Omega(\Phi_j(z_0),\Phi_j(z_0))}\to 0$ 
as $j\to \infty$ because $B_\Omega(z,z)\to \infty$ as $z\to q\in \partial\Omega$ 
(see \cite[Theorem 6.1.17]{JarnickiPflugBook} and \cite{Pflug75}) and $h_{\delta}$ 
is bounded. Since $\ep$ was arbitrary we conclude that 
$\lim_{z\to q}\langle h,f_j \rangle= 0$ for any $h\in A^2(\D)$. That is, $f_j\to 0$ 
weakly as $j\to \infty$ and the proof of Theorem \ref{thm-symmetry} is complete.
\end{proof}

\section{Membership in Schatten classes}\label{sec-schatten}
In this section we consider the problem of computing the Schatten $p$-norms 
of $R_U$ and $T_U$, provided we know the Bergman kernel 
$B_\Omega:\Omega\times\Omega\to \cx$  of the domain $\Omega$. 
Let us define the functions 
$B^{(p)}_{U,\Omega}:\Omega\times\Omega\to \cx$ for all $p\geq 1$, 
by setting $B^{(1)}_{U,\Omega}= B_\Omega,$ the Bergman kernel 
of $\Omega$, and (since $T_U$ is a positive operator)
\begin{equation}\label{eq-bp}
	B^{(p)}_{U,\Omega}(.,z)=T^{p-1}_UB_{\Omega}(.,z)
\end{equation}
 for $p>1$ and $z\in \D$. We note that 
 $B^{(p)}_{U,\Omega}(.,z) \in A^2(\D)$ for all $p\geq 1$ and $z\in \D$.  
 We have the following formula to compute the Schatten norms of 
 $R_U$ and $T_U$. 
\begin{proposition}\label{prop-schatten}
	Let $\Omega$ be a  domain in $\cx^n$ with nontrivial Bergman space, 
	$U$ be a non-empty open subset of $\Omega$, and $p>0$. 
	Then $R_U$ is a Schatten $p$-class operator if and only if 
	$\int_{\D}B^{(2p+1)}_{U,\Omega}(z,z)dV(z)<\infty$. Furthermore,    
\begin{align}\label{eq-pnorm}
\|R_U\|_{S_p}=\|T_U\|^{1/2}_{S_{2p}}
=\left(\int_{\D}B^{(2p+1)}_{U,\Omega}(z,z)dV(z)\right)^{1/4p}.
\end{align}
\end{proposition}

\begin{proof}[Proof of Proposition \ref{prop-schatten}]
Since $T_U$ is a positive operator $T_U^p$ is a positive positive operator 
on $A^2(\D)$ for every positive real  number $p$. 
Then for $z\in \D$ and $p>0$ we have 
\begin{align}\label{EqnSchatten}
0\leq \langle T^p_UB_{\D}(.,z),B_{\D}(.,z)\rangle 
=T^p_UB_{\D}(.,z)(z)=B^{(p+1)}_{U,\Omega}(z,z).
\end{align}
It is known (cf.  the proof in \cite[Theorem 6.4]{zhu_book_2007}) 
that the Schatten 1-norm of a positive operator 
$T$ on $A^2(\Omega)$ is given by:
\[\norm{T}_{S_1}=\int_{\D} \langle TB_{\D}(.,z),B_{\D}(.,z)\rangle dV(z),\]
where the integrand is equal to the Berezin transform of the operator $T$ 
times $B_{\D}(z,z)$. Note that both sides of the above equation may be infinite.  
Using \cite[Lemma 1.25]{zhu_book_2007} to compute $\norm{R_U}_{S_p}$ 
in terms of $\|T_U^{2p}\|_{S_1}$, we have
\begin{align*}
\norm{R_U}_{S_p}^{4p} =\norm{T_U}^{2p}_{S_{2p}}
&= \norm{T_U^{2p}}_{S_1}\\
&= \int_{\D} \langle T_U^{2p}B_{\D}(.,z),B_{\D}(.,z)\rangle dV(z),\\
&= \int_\Omega B^{(2p+1)}_{U,\Omega}(z,z) dV(z),
\end{align*}
where the last line follows from \eqref{EqnSchatten}. 
\end{proof}

We end the paper by  justifying the trace formula \eqref{eq-trace}.
\begin{align*}\
\norm{T_U}_{S_1} =\int_{\D}B^{(2)}_{U,\Omega}(z,z)dV(z)
=&\int_{\D}\int_{\D}\chi_{U}(\xi)B_{\D}(\xi,z)B_{\D}(z,\xi)dV(\xi)dV(z) \\
=&\int_{\D}\int_{U}|B_\Omega(z,\xi)|^2dV(\xi)dV(z)\\
 \text{(Using Tonelli's theorem)}
=&\int_{U}\int_{\D}|B_\Omega(\xi,z)|^2dV(\xi)dV(z) \\
=&\int_U B_\Omega(\xi,\xi) dV(\xi).
\end{align*}

\providecommand{\bysame}{\leavevmode\hbox to3em{\hrulefill}\thinspace}
\providecommand{\MR}{\relax\ifhmode\unskip\space\fi MR }
\providecommand{\MRhref}[2]{%
  \href{http://www.ams.org/mathscinet-getitem?mr=#1}{#2}
}
\providecommand{\href}[2]{#2}

\end{document}